\documentclass[12pt]{amsart}

\usepackage{amsmath,amssymb,amsthm,bm,dsfont,enumitem}

\usepackage[colorlinks,citecolor=blue,urlcolor=blue]{hyperref}  
 
\theoremstyle{plain}
\newtheorem{thm}{Theorem}[section]
\newtheorem{lem}[thm]{Lemma}
\newtheorem{prop}[thm]{Proposition}

\theoremstyle{definition}
\newtheorem{defn}{Definition}[section]

\newtheorem{exmp}{Example}[section]
\newtheorem{notation}{Notation}[section]
\newtheorem{assumption}{Assumption}[section]

\theoremstyle{remark}
\newtheorem*{rem}{Remark}

\newcommand\E{\mathcal{E}}
\newcommand\D{\mathcal{D}}

\newcommand\N{\mathbb{N}}
\newcommand\R{\mathbb{R}}

\newcommand\dx{\mathrm{d}x }
\newcommand\dy{\mathrm{d}y }

\newcommand\dt{\mathrm{d}t }
\newcommand\loc{\mathrm{loc}}
\newcommand\td{\mathrm{d} }

\newcommand\Euc{\!\mathrm{Euc}}

\newcommand\abs[1]{\lvert#1\rvert}
\newcommand\labs[1]{\left\lvert#1\right\rvert}
\newcommand\norm[1]{\lVert#1\rVert}

\DeclareMathOperator*{\divi}{div}
\DeclareMathOperator{\supp}{supp}
\DeclareMathOperator{\dist}{dist}

\DeclareMathOperator{\Lie}{Lie}

\DeclareMathOperator*{\inner}{Inner}

\thanks{Both author's research was supported by the Australian Research
  Council grant DP200101065. In addition, the first author's research
  was supported by the Australian Research
  Council grant DP220100067. Both authors are grateful for the support.}

\subjclass[2020]{58J35, 47H06, 47H20, 47H06, 35J92.}

\keywords{p-Laplace operator, $L^1$-$L^{\infty}$-estimates, separation, sub-Riemannian,
  Gru\v{s}in, harmonic analysis, nonlinear semigroups}

\begin{document}

\title{Regularity and Separation for $p$-Laplace operators}
\author{Daniel Hauer and Adam Sikora}

\address{Daniel Hauer, School of Mathematics and Statistics, The University of Sydney, Sydney, NSW, 2006, Australia }
\email{daniel.hauer@sydney.edu.au}
\address{
	Adam Sikora, School of Mathematical and Physical Sciences, Macquarie University, NSW 2109, Australia}
\email[Corresponding author]{adam.sikora@mq.edu.au }

\date{\today}

\begin{abstract}
  We analyze $p$-Laplace operators with degenerate elliptic
  coefficients. This investigation includes Gru\v{s}in type $p$-Laplace
  operators. We describe a \emph{separation phenomenon} in elliptic and parabolic 
  $p$-Laplace type equations, which provides an illuminating illustration of
  simple jump discontinuities of the corresponding weak solutions. Interestingly
  validity of an isoperimetric inequality for considered setting does
  not imply continuity of elliptic equations. On the other hand, we are
  able to establish global $L^1$-$L^\infty$-regularization and decay
  estimates of every mild solution of the parabolic Gru\v{s}in type $p$-Laplace equation.
\end{abstract}

\maketitle

{\it In memoriam of Derek W. Robinson.}

\section{Introduction}

Studies of problems involving the $p$-Laplace operator are a central and
crucial part of the theory of Partial Differential Equations
(PDEs). Such operators are second order non-linear generalizations of
the standard Laplace operator, which is arguably fundamental and the
most significant linear elliptic operator. The $p$-Laplace equation has
attracted intense attention during the last fifty or so years and its
theory is by now developed and well understood, see
e.g. \cite{LiP,MR1230384,MR2865434,MR2722059} and reference within.

Our main interest in this project is related to $p$-Laplace operators
corresponding to degenerate elliptic systems and geometry. It is a
worthy attention point because $p$-Laplace operators is of degenerate or 
singular form themselves. Indeed, in the Euclidean $\R^d$ setting these
operators can be written down as
\begin{displaymath}
  \Delta_{p}=\nabla\cdot\left(\abs{\nabla f}^{p-2}\nabla f\right),
\end{displaymath}
where $f\in C_c^\infty(\R^d)$ and $\nabla f$ is the standard Euclidean gradient,
compare \eqref{eq:p-A-Laplacian} below.  Clearly, for $p>2$ this
divergence equation is degenerate when $|\nabla f|=0$, and becomes
singular in region where $\nabla f$ vanishes. Hence considering
the degenerate elliptic system and the corresponding gradients for
$p$-Laplace operators is especially interesting, which it plays a
crucial role in nonlinear potential theory (see, e.g., \cite{MR1207810}).

We specifically investigate (generalized) Gru\v{s}in type $p$-Laplace
operators. Classical Gru\-\v{s}hin-type operators are a class of linear
degenerate elliptic operators (cf. Example~\ref{ex:1}~\ref{ex:1-classical-grushin})
that arise in the study of PDEs and functional inequalities (for
examples, in the Euler equation associated associated with a weighted
Sobolev inequality, see, e.g., \cite{MR2273962} or
\cite{MR2778606}). These operators often have specific forms and
properties that make them significant objects of our study. In the
classical Gru\v{s}in space setting, the linear case $p=2$ has been
studied extensively (see, e.g., \cite{MR2219239,RSik,Rsik2,DS}), but
only a few results have been worked out in the general case $1<p<\infty$
(see, e.g., \cite{MR2240671,MR4598995}). In \emph{generalized Gru\v{s}in
  space} setting (see Example~\ref{ex:1}~\ref{ex:gen-grusin} below),
even in the linear case $p=2$ only a few but crucial results are known;
these include the validity of a Nash inequality (see~\cite{RSik}) and
the failure of a local $L^2$-Poincar\'e inequality
(see~\cite{MR3250798}). To the best of our knowledge, similar results
for general case $1<p<\infty$ in this setting did not exist in the
literature so far.\medskip

In this article, we aim to change this by proving Nash and Sobolev
inequalities in the generalized Gru\v{s}in spaces setting for $p\ge 2$
(see Theorem~\ref{thm:Grushin-Nash-inequality} in
Section~\ref{subsec:sub-elliptic-Grushin}). For this, we employ
techniques developed in linear harmonic analysis to investigate
nonlinear $p$-Laplace equations of elliptic and parabolic type. A
valuable example of such approach is described, for instance, by Bortz,
Egert and Saari in \cite{BEO19} or \cite{MR4242322}. The $L^p$-Sobolev
inequality is a major tool to achieve $L^q$-$L^\infty$ regularization
effect, $1\le q<\infty$, for the semigroup generated by the
$p$-Laplacian, see \cite[Theorem~1.1]{CoulHau2017} and Theorem
\ref{thm65} below. It is worth mentioning that the degeneracy of the
generalized Gru\v{s}in space occurs in a higher dimension $D\ge d$,
which decreases the Sobolev-exponent $p^{\ast}=Dp/(D-p)$. Obviously, our
results generalize the known ones obtained in the classical Euclidean case
by V\'eron~\cite{MR554377}. For a relevant discussion of Sobolev
inequalities and the regularizing effect of solutions of parabolic
$p$-Laplace equations see, e.g., \cite{CoulHau2017}.

To introduce the geometry in our nonlinear PDEs, we first outline the
notion of degenerate Riemannian structures $(\R^d,\bm{A})$ on the
Euclidean space $\R^d$, $d\in \N$, see
Section~\ref{sec:sub-riemannian}. Here, $\bm{A}$ is the bundle of
symmetric endomorphisms, which defines 
the \emph{horizontal gradient} $\nabla_{\!\bm{A}(x)}f$ by setting
$\nabla_{\!\bm{A}(x)}f(x):=\bm{A}(x)\nabla_{\Euc}f(x)$ for every $x\in
\R^d$. Note that the matrix $\bm{A}(x)$ might degenerate at certain $x\in
\R^d$. This framework is sufficient to
introduce several sub-Riemannian spaces including the classical
Gru\v{s}in space and the Heisenberg group, or the generalized Gru\v{s}in
space, or spaces with monomial weights (see Example~\ref{ex:1} in
Section~\ref{subsec:examples}).

In Section~\ref{sec:spaces}, we introduce Soboleve spaces involving the horizontal gradient
$\nabla_{\!\bm{A}(x)}f$ in a weak sense. Under some reasonable assumptions,
we show that these spaces are complete, reflexive, and separable. These properties
are fundamental for realizing the \emph{horizontal $p$-Laplace operator}
\begin{displaymath}
  \Delta_p^{\bm{A}}f=\nabla\cdot (\abs{\nabla_{\!\bm{A}}f}_{\bm{A}}^{p-2}\nabla_{\!\bm{A}}f)
\end{displaymath}
associated with a given Riemannian structure $(\R^d,\bm{A})$. In Section
\ref{sec:carre-du-champ}, we expain how to employ the notion of the
\emph{carr\'e du champ} to highlight the underlying geometric structure
in the nonlinear equations governed by $\Delta_p^{\bm{A}}$. 
The carr\'e du champ is the significant notion in Dirichlet forms
theory, see e.g.~\cite{BH}, and another example for that we use tools
based on linear harmonic analysis theory to discuss ,

In the generalized Gru\v{s}in space setting, we study a \emph{phenomenon
  of separation}, which provides, a geometrically unexpected,
illuminating illustration of the limit of regularity results. An example
of such a phenomenon can be described even in one dimension. For
functions $f\in C^\infty_c(\R)$ we consider gradient
$\nabla_\alpha f(x) = |x|^\alpha f'(x)$, see Example \ref{ex41}
below. For any $0 \le \alpha <2$ the matching distance is finite for any
two points. However, for any $p$ such that $\alpha > 2/p'$, where
$p'=p/(p-1)$, the corresponding $p$-Laplace equation separates into two
totally independent systems acting on $(-\infty,0)$ and $(0,\infty)$,
see Lemma~\ref{lem:5} below. It follows that, in general, corresponding
elliptic or parabolic solutions are not necessarily continuous at the
point $x=0$. This is an interesting generalization of the separation
phenomenon for linear operators corresponding to the case $p$=2, see
\cite[Propositions 6.5 and 6.12]{ERSZ} and \cite[Proposition
6.10]{Rsik2}. A significant side of the separation phenomenon is that it
shows that the Muckenhoupt weights class order requirement for
continuity of $p$-Laplace equation solutions verified in
\cite{MR1207810} is sharp, see Section \ref{sec:separation-elliptic}
below. It is interesting to point out that the $L^1$-$L^{\infty}$
regularization effect of the semigroup generated by the Gru\v{s}in type
$p$-Laplace operator is not affected by the separation phenomenon.

We note that a similar \emph{phenomenon of separation} happens in spaces
of monomial weights (see Example~\ref{ex:1}~\ref{ex:monomial-weight}),
for which one knows that an isoperimetric inequality holds true
(cf.~\cite{MR3097258}), implying an $L^1$-$L^{\infty}$
regularization effect of the semigroup generated by the associated
$p$-Laplace operator in this setting (see Theorem~\ref{thm66-monomial}). It
a significant observation that validity of an isoperimetric inequality
does not imply continuity of elliptic equations.  For a description of
relevant characteristic of isoperimetric inequality to our discussion
see \cite{MR3097258}, see also \cite{CoSC}.

In this part of our project, the obtained results can be likely extended
but these developments will be investigate in some future extension of
this project.

\section{Some preliminaries - our framework} 

The purpose of this section is to outline how to introduce the geometry
into the elliptic and parabolic $p$-Laplace type equations. For this, we
use the notion of Riemannian structures $(\R^d,\bm{A})$. 

\subsection{A Riemannian structure}
\label{sec:sub-riemannian}

For $d\in \N$, let the pair $(\R^d,\bm{A})$ consisting of the Euclidean
space $\R^d$ and a symmetric, positive semi-definite matrix function
\begin{displaymath}
    \bm{A}\in L^{\infty}_{\loc}(\R^d;\R^{d\times d}).    
\end{displaymath}
In the following, we refer to such a pair as a \emph{Riemannian
  structure}.\medskip

In order to explain the class of degenerate 
diffusion operators of \emph{$p$-Laplace type}
\begin{displaymath}
   \Delta_{p}^{\!\bm{A}}f
:=\nabla\cdot\left(\abs{\nabla_{\!\bm{A}}f}_{\bm{A}}^{p-2}\nabla_{\!\bm{A}}f\right),
\end{displaymath}
induced by a Riemannian structure $(\R^d,\bm{A})$, we briefly summarize
the most relevant notions from Riemannian geometry. To do this, we begin
by considering the regular case.

\subsubsection{The elliptic setting.}
If the matrix $\bm{A}$ is assumed to be positive-definite on $\R^d$,
written $\bm{A}(x)>0$ for $x\in \R^d$, or more precisely,
\begin{equation}
    \label{eq:positive-definiteness}
 \xi^T\bm{A}(x)\xi>0\qquad\text{for all 
$x\in \R^d$, $\xi\in\R^d\setminus\{0\}$.}      
\end{equation}
Then, the pair $(\R^{d},\bm{g})$ consisting of the Euclidean space
$\R^d$ and the function $\bm{g} : \R^d\to \inner(\R^d)$ given by
\begin{equation}
    \label{eq:def-of-g}
  \bm{g}(x)=\langle\cdot, \cdot\rangle_{\bm{A}(x)}\qquad
  \text{for every $x\in \R^d$,}
\end{equation}
forms a classical $d$-dimensional Riemannian manifold with (Riemannian)
\emph{metric} $\bm{g}$. Here $\inner(\R^d)$ denotes the set of inner
products on $\R^d$. The manifold $(\R^{d},\bm{g})$ is determined through
the matrix $\bm{A}$. Note, in \eqref{eq:def-of-g} the inner products
$\langle\cdot, \cdot\rangle_{\bm{A}(x)}$ are defined by
\begin{equation}
    \label{eq:9}
   \langle X, Y\rangle_{\bm{A}(x)}:= \langle \bm{A}^{-1}(x)X,Y\rangle_{\Euc} 
\end{equation}
for every $X$, $Y\in \R^{d}$, and for $X$, $Y\in \R^d$, the bracket
$\langle X,Y \rangle_{\Euc}:=\sum_{i=1}^{d}X_i\overline{Y}_i$ refers
to the classical Euclidean inner product on $\R^{d}$, and we write
$\abs{X}$  or $\abs{X}_{\Euc}=\sqrt{\langle X,X \rangle_{\Euc}}$ to
denote the induced
Euclidean norm on $\R^d$.

\begin{rem}
  It is worth mentioning that in the framework outlined here, the
  Euclidean inner product $\langle \cdot,\cdot \rangle_{\Euc}$ could
  be replaced by any other inner product provided by some
  $d$-dimensional Riemannian manifold.
\end{rem}

\begin{notation}\label{not:square-root}
  The associated \emph{norm} $\abs{\cdot}_{\bm{A}}$ induced by the
  matrix $\bm{A}$ is then given by
\begin{equation}
 \label{eq:10}
  \abs{X}_{\bm{A}}=\sqrt{\langle X,X\rangle_{\bm{A}}}= \abs{\bm{B}^{-1}X}_{\Euc}
\end{equation}
where the matrix $\bm{B}$ is the so-called \emph{square root} of
$\bm{A}$, meaning $\bm{B}$ is the unique matrix of $\bm{A}$ satisfying
\begin{equation}
    \label{eq:11}
  \bm{B}^2(x)=\bm{A}(x)
\end{equation}
for every $x\in \R^d$. Note, if $\bm{A}(x)>0$ then also $\bm{B}$ satisfies $\bm{B}(x)>0$.
\end{notation}

\subsubsection{The gradient associated with a Riemannian structure} Next,
we introduce the notion of the gradient associated with a Riemannian
structure $(\R^d,\bm{A})$. For this note, we write either
$\nabla_{\Euc}$ or, simply, $\nabla$ and
$\tfrac{\partial }{\partial x_i}$ in order to denote the \emph{gradient}
and the \emph{partial derivatives} with respect to standard \emph{Euclidean} coordinates.

\begin{defn}
 For a given vector-field $X$ of $\R^d$, that is,
 \begin{displaymath}
    X=\sum_{i=1}^{d}\alpha_{i}\tfrac{\partial }{\partial x_i}
 \end{displaymath}
 for some coefficients $\alpha_{1}, \dots, \alpha_{d}\in \R$, and for every 
 $f\in C^{1}$, the \emph{action of $X$ on $f$ at $x\in \R^d$} is defined by 
 \begin{displaymath}
   Xf_{\vert x}=\sum_{i=1}^{d} \alpha_i \frac{\partial f}{\partial x_i}(x).
 \end{displaymath}
\end{defn}

\begin{rem}
 For given every $f\in C^{1}$ and vector-field $X$ of $\R^d$,
 the action $Xf_{\vert x}$ in \emph{local coordinates} can be
 interpreted as 
 \begin{equation}
    \label{eq:action-local}
  Xf_{\vert x}= \langle  \nabla f(x), X \rangle_{\Euc} 
 \end{equation}    
 Note, throughout this section, when we speak about \emph{local
   coordinates} then we mean \emph{Euclidean coordinates}.
\end{rem}

\begin{defn}
  For a given Riemannian structure $(\R^{d},\bm{A})$ and
  $f\in C^{1}(\R^{d})$, we call $\nabla_{\!\bm{A}}f$ defined by
\begin{equation}
 \label{eq:7}
 \langle  \nabla_{\!\bm{A}}f(x), X \rangle_{\bm{A}(x)}=Xf_{\vert x}
\end{equation}
for every vector-field $X$ of $\R^d$ and $x\in \R^{d}$, the
\emph{gradient of $f$ associated with the Riemannian structure
  $(\R^{d},\bm{A})$}, or shorter, \emph{gradient of $f$ associated with
  $\bm{A}$.}
\end{defn}

Applying \eqref{eq:action-local} to \eqref{eq:7}, one sees that for
given $f\in C^{1}$, the gradient $\nabla_{\!\bm{A}}f$ in local
coordinates can be characterized by
 \begin{displaymath}
  \langle \nabla_{\!\bm{A}}f, X \rangle_{\bm{A}}=
  \langle \nabla f, X \rangle_{\Euc}  
 \end{displaymath}
for every vector field $X\in \R^{d}$ and so,~\eqref{eq:9} yields that
\begin{displaymath}
  \langle \bm{A}^{-1}\nabla_{\!\bm{A}}f,X\rangle_{\Euc}
  = 
\langle  \nabla f, X \rangle_{\Euc}
\end{displaymath}
for every vector field $X\in \R^{d}$
. Therefore, we have the following description of the gradient
$\nabla_{\!\bm{A}}f$ in local coordinates.

\begin{prop}[{The Gradient in local coordinates}]
 For a given Riemannian structure $(\R^{d},\bm{A})$ and $f\in
 C^{1}(\R^d)$, the \emph{gradient} $\nabla_{\!\bm{A}}f$ in local
 coordinates can be rewritten by
\begin{equation}
 \label{eq:def-gradient-in-singular-case}
    \nabla_{\!\bm{A}}f= \bm{A}\nabla f
\end{equation}
\end{prop}

Now, by the definition \eqref{eq:9} of the inner products
$\langle\cdot, \cdot\rangle_{\bm{A}}$, one has that
\begin{align*}
\abs{\nabla_{\!\!\bm{A}}f}^2_{\bm{A}}
&=\langle  \nabla_{\!\bm{A}}f,\nabla_{\!\bm{A}}f\rangle_{\bm{A}}\\
&=\langle \bm{A}^{-1}(\bm{A}\nabla f),
\bm{A}\nabla f\rangle_{\Euc}\\
&=\langle \nabla f,
\bm{A}\nabla f\rangle_{\Euc}\\ 
&=\langle \bm{B}\nabla f,
\bm{B}\nabla f\rangle_{\Euc}\\
&= \abs{\bm{B}\nabla f}^{2}_{\Euc}
\end{align*}
for every $f\in C^{1}$. This shows, we have the following
characterization of the length of the gradient $\nabla_{\!\!\bm{A}}f$.

\begin{prop}
  For a given Riemannian structure $(\R^{d},\bm{A})$ and $f\in C^{1}$,
  one has that the \emph{length} (associated with $\bm{A}$)
  $\abs{\nabla_{\!\bm{A}}f}_{\bm{A}}$ of the gradient
  $\nabla_{\!\bm{A}}f$ is equivalent to
\begin{equation}
    \label{eq:12bis}
    \abs{\nabla_{\!\bm{A}}f}_{\bm{A}}
    =\sqrt{\langle \nabla f,
\bm{A}\nabla f\rangle_{\Euc}}
\end{equation}
Moreover, in local coordinates 
\begin{equation}
    \label{eq:12}
    \abs{\nabla_{\!\bm{A}}f}_{\bm{A}}
    =\abs{\bm{B}\nabla f}_{\Euc}=\sqrt{\langle \bm{B}\nabla f,
      \bm{B}\nabla f\rangle_{\Euc}}.
\end{equation}
\end{prop}

For convenience, we introduce the following notation.

\begin{notation}
 For a given Riemannian structure $(\R^{d},\bm{A})$ and $f\in C^{1}$, we set
 \begin{displaymath}
   \nabla_{\!\bm{B}}f=\bm{B}\nabla f,
 \end{displaymath}
 where the matrix $\bm{B}$ is the square root of $\bm{A}$. 
\end{notation}

The characterization \eqref{eq:12} of the length
$\abs{\nabla_{\!\bm{A}}f}_{\bm{A}}$ in local coordinates in terms of the
square root $\bm{B}$ of $\bm{A}$ has the advantage to relax the
requirement $\bm{A}$ being invertible.\medskip

With this preliminaries, we can now turn to the situation when the
matrix $\bm$ is degenerate.

\subsubsection{The degenerate setting}

From here and the rest of this section, let $(\R^{d},\bm{A})$ be a
Riemannian structure equipped with the following class of degenerate
matrix functions $\bm{A}$. Next for a connected subset
$K\subseteq \R^{d}$ we set up the distance $d_K(x)$ from $K$ for all
$x\in \R^{d}$ by the formula
\begin{equation}\label{dk}
d_K(x)=\inf_{y\in K}|x-y|.
\end{equation}

\begin{assumption}
\label{ass:1}
Suppose, $\bm{A}=(a_{ij})_{i,j=1}^{d}$ is a symmetric, positive
semi-definite matrix with coefficients $a_{ij}\in L^{\infty}_{loc}(\R^{d})$, 
$1 \le i,j \le n$, which we allow to
\emph{degenerate} in the following sense.\medskip 

\indent \emph{For integers $0\le n\le d$, let $K_1$, $\dots$, $K_n$ be
  connected subsets of $\R^d$ and of Lebesgue measure $\mathcal{L}^{d}$
  zero. Moreover, if any $K_i$ is not a one-point set, then we assume
  that $K_i$ is the graph of at least a $C^{1}$ function. Further,
  suppose $\bm{A}$ satisfies
\begin{equation}
 \label{eq:1}
 \xi^T\bm{A}(x)\xi=\sum_{i,j} a_{ij}(x)\xi_i \xi_j\ge \omega(x) |\xi|^2
\end{equation}
for a.e. $x\in \R^{d}$ and every $\xi\in \R^{d}$, where 
\begin{displaymath}
\omega (x) =\min\{d^{\gamma_1}_{K_1}(x), \dots, d^{\gamma_n}_{K_n}(x),1\}    
\end{displaymath}
for some  $\gamma_1, \dots, \gamma_n>0$.}
\end{assumption}

It is worth noting that under Assumption \ref{ass:1}, the matrix
$\bm{A}(x)$ has not to be invertible for $x\in K_i$. Hence, the inner
product $\langle\cdot, \cdot\rangle_{\bm{A}}$ associated with $\bm{A}$
might become \emph{singular} in $K_i$. On the other hand, the associated
gradient $\nabla_{\! \bm{A}}$ defined
by~\eqref{eq:def-gradient-in-singular-case} remains valid. Moreover,
there is still a unique positive semi-definite matrix $\bm{B}$
associated with $\bm{A}$ via~\eqref{eq:11}. Hence, \eqref{eq:12} can
still be used to define the length $\abs{\nabla_{\!\bm{A}}f}_{\bm{A}}$
of the gradient $\nabla_{\!\bm{A}}f$.\medskip


\subsection{Some examples of degenerate Riemannian structures}
\label{subsec:examples}

We conclude this preliminary section by illustrating the above setting
with some concrete examples.

\pagebreak

\begin{exmp}\label{ex:1}
\mbox{}
 \begin{enumerate}[label={\roman*.)}]
%
%
%
     \item \label{ex:1-classical-grushin} (\emph{Classical Gru\v{s}in space}) 
     Originally, Gru\v{s}in \cite{Gru} 
      introduced  a class of linear degenerate operators 
      \begin{displaymath}
        G_k=\partial_x^2 + |x|^{2k}\partial_y^2   
      \end{displaymath}
      defined for any $k \in \N$ and $(x,y)\in \R^{2}$.  This definition
      can easily be extended to the following larger class of operators
      with similar properties. For integers $d$, $m$, $n\ge 0$, let
      $d=m+n\ge 1$, $\beta_1$, \dots, $\beta_m\ge 0$, and $\bm{A}$ be
      the $d\times d$-diagonal matrix given by
     \begin{displaymath}
       \bm{A}=\begin{bmatrix} 
       \bm{I}_{n\times n} & 0 \\
       0 & \begin{bmatrix}
       |x|^{\beta_{1}}_{n} & 0 & \cdots & 0\\
       0& |x|^{\beta_{2}}_{n} & \ddots &  0\\
       \vdots & \ddots & \ddots & 0\\
       0 & \cdots &   0 & |x|^{\beta_{m}}_{n}
       \end{bmatrix}_{m\times m}
       \end{bmatrix},
     \end{displaymath}
     where $\bm{I}_{n\times n}$ refers to the \emph{identiy} matrix in
     $\R^{n\times n}$, and we set
      \begin{equation}
          \label{eq:n-norm}
          |x|_{n}:=\left(\sum_{i=1}^{n}x_{i}^{2}\right)^{1/2},\;
          x=(x_{1},\dots,x_{n},0,\dots,0)\in \R^{n+m}
      \end{equation}
      the Euclidean norm of the \emph{first $n$-coordinates}.\medskip
     
     For this matrix $\bm{A}$, the gradient $\nabla_{\!\bm{A}}f$ is given by
     \begin{displaymath}
        \nabla_{\!\bm{A}}f= 
        \begin{bmatrix} 
       \frac{\partial f}{\partial x_1},\dots,\frac{\partial f}{\partial x_n},
       |x|^{\beta_{1}}_{n}\frac{\partial f}{\partial y_1},
       \dots,  |x|^{\beta_{m}}_{n}\frac{\partial f}{\partial y_m}
       \end{bmatrix},
     \end{displaymath}
     and the length $\abs{\nabla_{\!\bm{A}}f}_{\bm{A}}$ in local coordinates is given by
     \begin{displaymath}
        \abs{\nabla_{\!\bm{A}}f}_{\bm{A}}^2=\sum_{i=1}^n\labs{\tfrac{\partial f}{\partial x_i}}^2
        +\sum_{j=1}^{m}\abs{x}_n^{\beta_{j}}\labs{\tfrac{\partial f}{\partial y_i}}^2.
     \end{displaymath}
     
     In sepcial case , for every $j=1,\dots, m$, let $k_{j}\in \N$,
     $\beta_{j} =2 k_j$, and $p_{j}(x)=\abs{x}^{\beta_{j}}$. Then, one
     can introduced as system of corresponding vector fields
     $X=\{X_{1},\dots, X_{d}\}$ by the formula
     \begin{displaymath}
       X_{j}=
       \begin{cases}
        \frac{\partial}{\partial x_{j}} & \text{for $j=1,\dots, n$,}\\
        p_{j-n}\frac{\partial}{\partial x_{j}} & \text{for $j=n+1,\dots, n+m=d$,}\\
       \end{cases}
     \end{displaymath}
     The system $X$ satisfies the \emph{H\"ormander condition}; that is,
     if $\Lie(X)$ denotes the Lie algebra generated by $X$, then one has
     that
     \begin{displaymath}
       \Lie_{x}(X):=\Big\{\tilde{X}(x)\,
       \vert\,\tilde{X}\in \Lie(X)\Big\}=\mathcal{T}_{q}(\mathds{G}^{n,m})\
     \end{displaymath}
     for every $x\in \R^{d}$. Here,
     $\mathcal{T}_q(\mathds{G}^{n,m})=:g_{n,m}$ denotes the tangent
     space of the \emph{Gru\v{s}in space} $\mathds{G}^{n,m}$ (or, also
     written as $\R^{n,m}$). Note that, in the consider setting,
     $\Lie(X)$ is finite dimensional nilpotent Lie algebra.\medskip

     An example of such system corresponding to the martix $A$ with
     $\alpha=0$ and $\beta_j=2N$ for some $N\in \N$ was studied by
     Robinson and the second author in \cite{Rsik2}. See also \cite{DS}
     for some other possible setting involving Lie algebra techniques.

   \item \label{ex:gen-grusin} (\emph{Generalized Gru\v{s}in space})
     Now, for integers $d$, $m$, $n\ge 0$, and for $0\le \alpha<2$, let
     $d=n+m\ge 1$, $\beta_1$, \dots, $\beta_m\ge 0$ and $\bm{A}$ be the
     $d\times d$-diagonal matrix given by
\begin{equation}
  \label{eq:genGrushin-Matrix}
       \bm{A}=\begin{bmatrix} 
       |x|^{\alpha}_n\,\bm{I}_{n\times n} & 0 \\
       0 & \begin{bmatrix}
       |x|^{\beta_{1}}_{n} & 0 & \cdots & 0\\
       0& |x|^{\beta_{2}}_{n} & \ddots &  0\\
       \vdots & \ddots & \ddots & 0\\
       0 & \cdots &   0 & |x|^{\beta_{m}}_{n}
       \end{bmatrix}_{m\times m}
       \end{bmatrix},
\end{equation}
     
where as before $|x|_{n}$ denotes the Euclidean norm \eqref{eq:n-norm}
for the first $n$-coordinates
$x=(x_{1},\dots,x_{n},0,\dots,0)\in \R^{d}$.\medskip
    
For this matrix $\bm{A}$, the gradient $\nabla_{\!\bm{A}}f$ is given by
     \begin{displaymath}
        \nabla_{\!\bm{A}}f= 
        \begin{bmatrix} 
       |x|^{\alpha}_{n}\frac{\partial f}{\partial x_1},
       \dots,|x|^{\alpha}_{n}\frac{\partial f}{\partial x_n},
       |x|^{\beta_{1}}_{n}\frac{\partial f}{\partial y_1},
       \dots,  |x|^{\beta_{m}}_{n}\frac{\partial f}{\partial y_m}
       \end{bmatrix},
     \end{displaymath}
     and the length $\abs{\nabla_{\!\bm{A}}f}_{\bm{A}}$ in local
     coordinates is given by
     \begin{equation}
        \label{eq:Grushin-gradient-norm}
        \abs{\nabla_{\!\bm{A}}f}_{\bm{A}}^2=|x|^{\alpha}_{n}
        \sum_{i=1}^n\labs{\tfrac{\partial f}{\partial x_i}}^2
        +\sum_{j=1}^{m}\abs{x}_n^{\beta_{j}}
        \labs{\tfrac{\partial f}{\partial y_i}}^2.
     \end{equation}
     Note, the case $\alpha =0$ is considered in the classical
     Gru\v{s}in type spaces mentioned above. In Section
     \ref{sec:Separation}, we consider the case $d=1+m$ for
     $0\le \alpha<2$ and $\beta_1$, \dots, $\beta_m>0$. In this case,
     the structure $\bm{A}$ from \eqref{eq:genGrushin-Matrix} reduces to
    \begin{displaymath}
  %
       \bm{A}=\begin{bmatrix} 
       |x_1|^{\alpha}\, & 0 \\
       0 & \begin{bmatrix}
       |x_1|^{\beta_{1}} & 0 & \cdots & 0\\
       0& |x_1|^{\beta_{2}} & \ddots &  0\\
       \vdots & \ddots & \ddots & 0\\
       0 & \cdots &   0 & |x_1|^{\beta_{m}}
       \end{bmatrix}_{m\times m}
       \end{bmatrix},
     \end{displaymath}
     and the length $\abs{\nabla_{\!\bm{A}}f}_{\bm{A}}$ in local
     coordinates is given by
   \begin{displaymath}
     \abs{\nabla_{\!\bm{A}}f}_{\bm{A}}^2
     =|x_1|^{\alpha}\labs{\tfrac{\partial f}{\partial x_1}}^2+
     \sum_{j=1}^{m}|x_1|^{\beta_j}\labs{\tfrac{\partial f}{\partial y_j}}^2.
     \end{displaymath}
     This structure $(\R^d,\bm{A})$ leads to a separation phenomenon for
     $2\le p<\infty$, $\frac{2}{p^{\mbox{}_{\prime}}} \le \alpha<2$.
    
     We emphasize that the case $\alpha\ge 2$ is geometrically not
     relevant because of geometric reasons. Indeed, consider a curve
     $\gamma \colon [a,b] \to \R$ connecting points $a,b\in \R$ defined
     by $\gamma(t)=t$. Then, obviously, $\gamma'(t)=1$ and according to
     \eqref{eq:10},
    \begin{displaymath}
        |\gamma'(x)|_{\bm{A}(x)}=\langle 1, 1 
        \rangle_{\bm{A}(x)}^{1/2}=|x|^{-\alpha/2}       
    \end{displaymath}
    for every $x\in [a,b]$. Hence, the distance
    \begin{equation}\label{dis}
        d_A(a,b)= \int_a^b|\gamma'(x)|_{\bm{A}(x)}\,dx=\int_a^b |x|^{-\alpha/2}dx  
     \end{equation}
     is finite for every $a,b \in \R$ if $0<\alpha<2$, and
     $d_A(0, a) = \infty$ for any $a\in \R$ if $\alpha \ge 2$.\medskip

   \item (\emph{Further generalization}) Let $d\ge 1$ and $K$ be a
     $(d-1)$-dimensional hyperplane in $\R^{d}$.  If $d_K$ is the
     distance defined in \eqref{dk} then, for given exponents
     $\alpha_{1}, \dots, \alpha_{d}\ge 0$, let $\bm{A}$ be the diagonal
     matrix defined by
     \begin{displaymath}
       \bm{A}= \begin{bmatrix}
       d_{K}^{\alpha_{1}}& 0 & 0 & 0 & \cdots & 0\\
       0& d_{K}^{\alpha_{2}} & 0 & 0 & \cdots & 0\\
       \vdots & 0 & d_{K}^{\alpha_{3}} & 0 & \cdots & 0\\
       0 & \cdots & \ddots & \ddots & \ddots & 0\\
       0 & \cdots & \ddots & \ddots & 0 & d_{K}^{\alpha_{d}}
       \end{bmatrix}.
     \end{displaymath}
     One easily verifies that the matrix $\bm{A}$ satisfies
     Assumption~\ref{ass:1} with exponent
     $\gamma=\min\{\alpha_{1}, \dots, \alpha_{d}\}$.

  \item  (\emph{Heisenberg Group}) Let $d=3$ and $\bm{A}$ be the matrix given by 
  \begin{displaymath}
       \bm{A}(x,y,z)= 
        \begin{bmatrix}
            1 & 0 & y/2\\
            0 & 1 & -x/2\\
            y/2 & -x/2 & (x^2+y^2)/2
        \end{bmatrix}.
    \end{displaymath}
    For this matrix, the gradient $\nabla_{\bm{A}}f$ is the, in fact, the \emph{horizontal gradient} associated with the \emph{Heisenberg group} (cf., for instance, \cite{MR2312336}). Note, the matrix $\bm{A}$ does not satisfy Assumption~\ref{ass:1} introduced in this paper. It is an interesting example of a sub-Riemannian structure $(\R^3,\bm{A})$, but goes beyond the discussion of this paper. See for example the monograph~\cite{MR3444525} by Ricciotti concerning elliptic regularity theory of the $p$-Laplace equation in the Heisenberg group, and we refer to the recent regularity result \cite{MR4611379} by Capogna, Citti, and Zhong for the parabolic $p$-Laplace equation in the Heisenberg group setting.

 \item \label{ex:monomial-weight} (\emph{Spaces with monomial weights}) Now, for integers $d\ge 1$, and for $0\le \alpha_1,\dots, \alpha_d <1$, let
    $\bm{A}$ be the $d\times d$-diagonal matrix given by
\begin{equation}
 \label{eq:monomilal}
      \bm{A}=\begin{bmatrix}
      |x_1|^{\alpha_{1}} & 0 & \cdots & 0\\
      0& |x_2|^{\alpha_{2}} & \ddots &  0\\
      \vdots & \ddots & \ddots & 0\\
      0 & \cdots &   0 & |x_d|^{\alpha_{d}}
      \end{bmatrix},
\end{equation}
    where as before $|\cdot|$ denotes the absolute value. It is easy
    to see that
    \begin{align*}
      \xi^T\bm{A}(x)\xi&=\sum_{i=1}^d |x_i|^{\alpha_{i}}
                         \abs{\xi_i}^2\\
      &\ge \min\{|x_1|^{\alpha_{1}},\dots, |x_d|^{\alpha_{d}},1\}\, |\xi|^2
    \end{align*}
    for every $\xi\in \R^d$. Thus, the matrix $\bm{A}$ satisfies Assumption~\ref{ass:1}, and the
    associated gradient $\nabla_{\!\bm{A}}f$ is given by
    \begin{displaymath}
       \nabla_{\!\bm{A}}f= 
       \begin{bmatrix} 
      |x_1|^{\alpha_1}\frac{\partial f}{\partial
  x_1},\dots, |x_d|^{\alpha_{d}}\frac{\partial f}{\partial x_d}
      \end{bmatrix}.
    \end{displaymath}
 This example of $(\R^d,\bm{A})$ was, for instance, mentioned in~\cite{MR3097258} by Cabr\'e and Ros-Oton and in connection with other Sobolev inequalities with monomial weights. It is worth mentioning that as the degenerate Riemannian structure $(\R^d,\bm{A})$ discussed in Example ii.), also $(\R^d,\bm{A})$ given by \eqref{eq:monomilal} leads to a \emph{separation phenomena} at $\{x_i=0\}$ for every $i=1, \dots, d$ provided $1\le \alpha_i<2$. For further details to this topic, we refer to Section~\ref{sec:Separation}.


   \item (\emph{Poincar\'e half-plane}) Let $d=2$ and $\bm{A}$ be the diagonal matrix
   \begin{displaymath}
     	\bm{A}=
        \begin{bmatrix} x^2 & 0\\
     	    0 & x^2
     	\end{bmatrix}.
     	\end{displaymath}
     	Then $\bm{A}$ is symmetric and satisfies hypothesis~\eqref{eq:1} with equality; namely, one has that
     	\begin{displaymath}
     		\xi^{t}\bm{A}\xi= x^{2}\xi_{1}^2+x^2\xi_{2}^{2}=x^2\abs{\xi}^2.
     	\end{displaymath}
     	Moreover, the associated gradient
     	\begin{displaymath}
     		\nabla_{\!\! \bm{A}}f=\begin{bmatrix} x^2 & 0\\
     		0 & x^2
     	\end{bmatrix}
     	\begin{bmatrix} f_{x}\\
     		f_{y}
     	\end{bmatrix}=
        \begin{bmatrix} 
            x^2\,f_{x} & x^2\,f_{y}
     	\end{bmatrix}.
     	\end{displaymath}
     		
     This example shows well that the matrix $\bm{A}$ decides on the length of the geodesics. In fact, the metric $\bm{g}$ induced by the matrix $\bm{A}$ (cf~\eqref{eq:def-of-g}) coincides with the metric of two copies of the \emph{Poincar\'e half-plane}.
     	
     
 \end{enumerate} 
\end{exmp}


\section{Lebesgue and mixed Sobolev spaces}
\label{sec:spaces}

Throughout this section, let $(\R^{d},\bm{A})$ be a Riemannian structure
equip\-ped with matrix functions $\bm{A}$ satisfying the Assumption
\ref{ass:1}.\medskip 

The aim of this section is to outline the Lebesgue and
$1^{\textrm{st}}$-order Sobolev spaces, which are used later on.
 
 \begin{notation}
    For $1\le q\le \infty$, and a given open subset $\Omega$ of $\R^d$, we denote by $L^{q}(\Omega)$ the standard Lebesgue
    space with respect to the Lebesgue measure restricted on an open subset
    $\Omega$ of $\R^{d}$, and by $\norm{f}_{q}$ the associated norm. Further, $W^{1,1}_{loc}(\Omega)$ refers to the space of all $f\in L^1_{loc}(\Omega)$ with distributional partial derivative $\frac{\partial f}{\partial x_i}\in L^1_{loc}(\Omega)$ for every $i=1, \dots, d$.
 \end{notation}

\begin{defn}
 Let $1\le p$, $q\le \infty$, and $\Omega$ an open subset of $\R^d$. Then, we denote by $W^{1,(q,p)}_{\!
 \bm{A}}(\Omega)$ the \emph{$1^{\textrm{st}}$-order mixed
  Sobolev space} of all
functions $f\in L^{q}(\Omega)\cap W^{1,1}_{loc}(\Omega)$ such that
$\abs{\nabla_{\! \bm{A}}f}_{\bm{A}}\in L^{p}(\Omega)$. We equip the space
$W^{1,(q,p)}_{\! \bm{A}}(\Omega)$ with the sum norm
\begin{displaymath}
  \norm{f}_{W^{1,(q,p)}_{\! \bm{A}}(\Omega)}=\norm{f}_{q}+\norm{\abs{\nabla_{\!
  \bm{A}}f}_{\bm{A}}}_{p}.
\end{displaymath}
Further, we denote by $W^{1,(q,p)}_{\! \bm{A},0}(\Omega)$ the closure of the set
of test-functions $C^{\infty}_{c}(\Omega)$ in $W^{1,(q,p)}_{\! \bm{A}}(\Omega)$.
\end{defn}

\begin{notation}
If $\bm{A}=\mathbb{I}_{d\times d}$ is the identity matrix in $\R^{d\times d}$, then we
simply write $W^{1,(q,p)}(\Omega)$ instead of
$W^{1,(q,p)}_{\! \mathbb{I}_{d\times d}}(\Omega)$. Further, if $q=p$ then we write
$W^{1,q}_{\bm{A}}(\Omega)$ instead of $W^{1,(q,q)}(\Omega)$. 
\end{notation}

\begin{prop}
 \label{prop:1}
 Let $1\le q\le \infty$ and $1\le p< \infty$. Then, the following statements
 hold.
 \begin{enumerate}[label={(\arabic*.)}]
     \item \label{prop:1-claim1} The $1^{\textrm{st}}$-order mixed
  Sobolev space $W^{1,(q,p)}_{\! \bm{A}}(\Omega)$ is a Banach space. In
  particular, $W^{1,2}_{\bm{A}}(\Omega)$ is a Hilbert space.
    \item \label{prop:1-claim2} For $1<q$, $p<\infty$, $W^{1,(q,p)}_{\!
    \bm{A}}(\Omega)$ is reflexive.
    \item \label{prop:1-claim3} For $1\le q$, $p<\infty$, $W^{1,(q,p)}_{\!
    \bm{A}}(\Omega)$ is separable.
 \end{enumerate}
\end{prop}

\begin{proof}
  We begin by showing that \ref{prop:1-claim1} holds. Let
  $(f_{n})_{n\ge 1}$ be a Cauchy sequence in
  $W^{1,(q,p)}_{\! \bm{A}}(\Omega)$. Since the space
  $W^{1,(q,p)}_{\!  \bm{A}}(\Omega)$ is continuously embedded into
  $L^q(\Omega)$, and since $L^{q}(\Omega)$ is complete, there is a
  function $f\in L^q(\Omega)$ such that $f_{n}\to f$ in
  $L^q(\Omega)$. Now, for every
  $i=1,\,\dots,\,n$, and $m\ge 1$, set
  \begin{equation}
    \label{eq:set-Km}
    \Omega_{m}=\Big\{x\in \Omega\,\Big\vert\, \min_{i=1,\dots,
      n}d_{K_i}(x)\ge 2^{-m},\;\abs{x}<2^{m}\Big\}.
  \end{equation}

  Let $\varepsilon>0$. Since by assumption, $(f_{n})_{n\ge 1}$ is a
  Cauchy sequence in $W^{1,(q,p)}_{\! \bm{A}}(\Omega)$, there is an
  $N=N(\varepsilon)\ge 1$ such that
  \begin{displaymath}
    \norm{\abs{\nabla_{\!\bm{A}}f_{n_{1}}-\nabla_{\!\bm{A}}f_{n_{2}}}_{\bm{A}}}_{p
    }<\varepsilon
  \end{displaymath}
  for every integer $n_{1}$, $n_{2}\ge N$ and so, by 
  \eqref{eq:12bis} and~\eqref{eq:1}, one see that
 \begin{equation}
    \label{eq:14}
    \tfrac{1}{2^{\frac{m}{2}p}}
    \int_{\Omega_{m}}\abs{\nabla(
    f_{n_{1}}- f_{n_{2}})}^{p}_{\Euc}\,
    \dx\le \int_{\Omega_{m}}\abs{\bm{B}\nabla( f_{n_{1}}-f_{n_{2}})}^{p}_{\Euc}\,\dx<\varepsilon^{p}
 \end{equation}
 for every integer $n_{1}$, $n_{2}\ge N$ and every $m\ge 1$. Since
 $W^{1,(q ,p)}(\Omega_{m})$ is complete, $f_{n}\to f$ in
 $L^{q}(\Omega)$, and since $\bm{B}$ belongs to
 $L^{\infty}(\Omega_{m};\R^{d\times d})$, we can conclude that
 $f\in W^{1,(q ,p)}(\Omega_{m})$, and
 $\bm{B}\nabla f_{n}\to\bm{B}\nabla f$ in
 $L^{p}(\Omega_{m};\R^{d})$. Therefore, sending
 $n_{1}\to \infty$ in
 \begin{displaymath}
   \int_{\Omega_{m}}\abs{\bm{B}\nabla( f_{n_{1}}-f_{n_{2}})}^{p}_{\Euc}\,\dx<\varepsilon^p
 \end{displaymath}
 yields that
 \begin{displaymath}
   \int_{\Omega_{m}}\abs{\bm{B}\nabla( f-f_{n_{2}})}^{p}_{\Euc}\,\dx\le\varepsilon^p
 \end{displaymath}
  for every $n_{2}\ge N$ and every $m\ge 1$. By construction, we have that
 \begin{displaymath}
   \Omega_{m}\subseteq \Omega_{m+1}
   \quad\text{ and }\quad
   \bigcup_{m\ge 1}\Omega_{m}=\Omega\setminus \bigcup_{i=1}^{n}K_i
 \end{displaymath}
 and $\bigcup_{i=1}^{n}K_i$ has Lebesgue measure $\mathcal{L}^d$ zero.
Thus, sending $m\to \infty$ in the last inequality lead to
 \begin{displaymath}
   \norm{\abs{\nabla (f-f_{n_{2}})}_{\bm{A}}}_{p}\le \varepsilon
 \end{displaymath}
 for every $n_{2}\ge N$, implying that
 $f\in W^{1,(q,p)}_{\! \bm{A}}(\Omega)$ and that
 $f_{n}\to f$ in $W^{1,(q,p)}_{\! \bm{A}}(\Omega)$. Since
 $(f_{n})_{n\ge 1}$ was an arbitrary Cauchy sequence, we have thereby
 shown that $W^{1,(q,p)}_{\! \bm{A}}(\Omega)$ is Banach space. In particular,
 $W^{1,2}_{\! \bm{A}}(\Omega)$ equipped with the inner product
 \begin{displaymath}
   \langle f,g \rangle_{W^{1,2}_{\! \bm{A}}(\Omega)}:=
   \int_{\Omega} f\,g \dx+\int_{\Omega}\langle 
   \nabla_{\!\bm{B}}f,\nabla_{\!\bm{B}}g \rangle_{\Euc}\dx
 \end{displaymath}
 for every $f$, $g\in W^{1,2}_{\! \bm{A}}(\Omega)$ is a Hilbert space.\medskip
 
 Next, we outline the proof of claim~\ref{prop:1-claim3} by following
 the standard argument (cf, \cite[Proposition~8.1]{MR2759829}). The
 product space $E:=L^{q}(\Omega)\times L^{p}(\Omega;\R^{d})$ is
 reflexive if $1<q$, $p<\infty$, and separable if $1<q$,
 $p<\infty$. Further, the mapping
 $T: W^{1,(q,p)}_{\! \bm{A}}(\Omega)\to E$ defined by
 $Tf:=(f,\nabla_{\!\bm{B}}f)$ is an isometry and by~\ref{prop:1-claim1},
 the image $T(W^{1,(q,p)}_{\! \bm{A}}(\Omega))$ is a closed subspace of
 $E$. Thus, the fact that the space $W^{1,(q,p)}_{\! \bm{A}}(\Omega)$ is
 reflexive and separable follows from the reflexivity and separability
 of $T(W^{1,(q,p)}_{\! \bm{A}}(\Omega))$.
\end{proof}

After introducing Sobolev spaces with respect to a Riemannian structure
$(\R^d,\bm{A})$, it be natural to mention valid Poincar\'e and Sobolev
inequalities in this setting. But, Assumption \ref{ass:1} is too general
for the validity of such inequalities.



\section{Degenerate p-Laplace operators}
\label{sec:plaplacian}

Let $(\R^{d},\bm{A})$ be a Riemannian structure equipped with matrix
functions $\bm{A}$ satisfying the Assumption \ref{ass:1}. Further, let
$\Omega$ be an open subset of $\R^{d}$ and $1<p<\infty$.\medskip

Then, the main object of this section is the 
so-called \emph{$p$-Laplace operator} 
\begin{equation}\label{eq:p-A-Laplacian}
   \Delta_{p}^{\!\bm{A}}f:=\nabla\cdot(\abs{\nabla_{\!\bm{A}}f
    }^{p-2}_{\bm{A}}\nabla_{\!\bm{A}}f)\qquad\text{in $\mathcal{D}'(\Omega)$}
 \end{equation}
associated with $(\R^d,\bm{A})$ for $f\in W^{1,(q,p)}_{\!\bm{A},loc}(\Omega)$, equipped with \emph{homogeneous Dirichlet boundary conditions} on $\partial\Omega$.\medskip

To realize the $2^{\textrm{nd}}$-order differential operator $\Delta_{p}^{\!\bm{A}}$, we consider the functional $\E : L^2(\Omega)\to [0,+\infty]$ given by
\begin{equation}
 \label{eq:3}
\E(f):=
 \begin{cases}
 \tfrac{1}{p}\displaystyle
 \int_{\Omega}\abs{\nabla_{\!\! \bm{A}}f}_{\bm{A}}^{p}\,\dx   
 &\text{if $f\in W^{1,(2,p)}_{\!\bm{A},0}(\Omega)$,}\\
 +\infty &\text{otherwise,}
 \end{cases}
\end{equation}
for every $f\in L^{2}(\Omega)$.\medskip

\begin{prop}
    \label{prop:2}
    Under the assumptions of this section, the functional $\E$ given
    by~\eqref{eq:3} is proper, densely defined, convex and lower
    semi-continuous on $L^{2}(\Omega)$.
\end{prop}

In the following, the matrix $\bm{B}$ denotes the square root of
$\bm{A}$ as introduced in Notation~\ref{not:square-root}.

\begin{proof}
  The bulk of the proof of this proposition has been done in
  Proposition~\ref{prop:1}.  Thus the reminding argument is quite
  standard (see, for
  instance,~\cite{CoulHau2017,MR3465809,zbMATH06275549,hauer07}). For
  the sake of completeness, we provide the details. Let $\alpha>0$ and
  $(f_{n})_{n\ge 1}$ be a sequence in
  $D(\E)=W^{1,(2,p)}_{\!\bm{A},0}(\Omega)$ converging in $L^{2}(\Omega)$
  to an element $f\in L^{2}(\Omega)$ and satisfying
  $\E(f_{n})\le \alpha$ for all $n\ge 1$. Since the sequence
  $(f_{n})_{n\ge 1}$ is bounded in $W^{1,(2,p)}_{\!\bm{A},0}(\Omega)$
  Proposition~\ref{prop:1} yields that
  $f\in W^{1,(2,p)}_{\!\bm{A},0}(\Omega)$ and $(f_{n})_{n\ge 1}$ admits
  a subsequence, which, for simplicity, we denote again by itself, such
  that $\nabla_{\! \bm{B}}f_{n}\rightharpoonup \nabla_{\! \bm{B}}f$
  weakly in $L^{p}(\Omega;\R^{d})$. Moreover, by \eqref{eq:12} one has
  that
  \begin{displaymath}
    \E(f)=\int_{\Omega}\abs{\nabla_{\! \bm{B}}f}_{\Euc}^{p}\dx\le 
    \liminf_{n\to\infty}\int_{\Omega}\abs{\nabla_{\!
        \bm{B}}f}_{\Euc}^{p}\dx\le \alpha.
  \end{displaymath}  
  Therefore, $\E$ is lower semi-continuous on $L^{2}(\Omega)$. The fact
  that $\E$ is densely defined follows from the fact that the set
  $C^{\infty}_{c}(\Omega)$ is contained in
  $D(\E)=W^{1,(2,p)}_{\!A,0}(\Omega)$ and dense in $L^{2}(\Omega)$.
\end{proof}

Next, we intend to compute the sub-gradient $\partial \E$ in
$L^{2}(\Omega)$.

\begin{prop}\label{prop:2bis}
  The restriction $\E_{\vert W^{1,(2,p)}_{\!\bm{A}, 0}}$ of the
  functional $\E$ given by~\eqref{eq:3} is of the class $C^{1}$ and its
  sub-gradient operator $\partial\E$ in $L^{2}(\Omega)$ is a
  well-defined mapping $\partial\E : D(\partial\E) \to L^{2}(\Omega)$
  given by
\begin{displaymath}
    \partial\E(f)=-\nabla\cdot(\abs{\nabla_{\!\bm{A}}f
    }^{p-2}_{\!\bm{A}}\nabla_{\!\bm{A}}f)\qquad\text{in
    $L^{2}(\Omega)$}
\end{displaymath}
for every $f\in D(\partial\E)$, where the domain
\begin{displaymath}
  D(\partial\E)=\Big\{f\in W^{1,(2,p)}_{\!\bm{A},
    0}(\Omega)\,\Big\vert\,
  \begin{array}[c]{c}
    \exists\,h\in L^2(\Omega)\text{ s.t. }\forall\,\xi
    \in W^{1,(2,p)}_{\!\bm{A},
    0}(\Omega),\\
    \langle \E'(f),\xi\rangle=\int_{\Omega}h\,\xi\,\dx\,
  \end{array}
\Big\}.
\end{displaymath}
\end{prop}

\begin{proof}
  Recall that the matrix $\bm{B}$ is symmetric and satisfies
  $\bm{B}^2=\bm{A}$. Moreover, one has that \eqref{eq:12bis} and
  \eqref{eq:12}. Thus, proceeding as in the classical case
  $\bm{A}=\mathbb{E}_{d\times d}$ the identity matrix in
  $\R^{d\times d}$ (see, for instance,~\cite{hauer07}) yields that the
  restriction $\E_{\vert W^{1,(2,p)}_{\!\bm{A}, 0}}$ of the functional
  $\E$ given by~\eqref{eq:3} is continuously differentiable with
  Fr\'echet-derivative
 \begin{align*}
     \langle \E'(f),\xi\rangle_{(W^{1,(2,p)}_{\!\bm{A}, 0})',
   W^{1,(2,p)}_{\!\bm{A}, 0}}&=\int_{\Omega}\abs{\nabla_{\!\bm{B}}
     f}^{p-2}_{\Euc}
     \langle \nabla_{\!\bm{B}}f,\nabla_{\!\bm{B}}\xi
     \rangle_{\Euc}\,\dx\\
     &=\int_{\Omega}\abs{\nabla_{\!\bm{A}}f}^{p-2}_{\bm{A}}
     \langle \nabla_{\!\bm{A}}f,\nabla\xi\rangle_{\Euc}
     \,\dx 
 \end{align*}
 for every $f\in W^{1,(2,p)}_{\!\bm{A},0}(\Omega)$. 
 
 Now, let $f\in D(\partial\E)$. Then,
 $f\in W^{1,(2,p)}_{\!\bm{A},0}(\Omega)$ and there is an element
 $h\in L^{2}(\Omega)$ satisfying
 \begin{equation}
     \label{eq:4}
     \langle h,g-f\rangle_{L^{2}(\Omega)}\le \E(g)-\E(f)\qquad\text{for every
     $g\in L^{2}(\Omega)$}
 \end{equation}
 Since $\E$ is convex and differentiable, choosing $g=f+t\xi$
 in~\eqref{eq:4} for given $t>0$ and
 $\xi\in W^{1,(2,p)}_{\!\bm{A},0}(\Omega)$ yields that
 \begin{displaymath}
    \langle h,\xi\rangle_{L^{2}(\Omega)}\le
    \inf_{t>0}\frac{\E(f+t\xi)-\E(f)}{t}=\langle \E'(f),\xi\rangle_{(W^{1,(2,p)}_{\!\bm{A}, 0})',
   W^{1,(2,p)}_{\!\bm{A}, 0}}.
 \end{displaymath}
 Further, taking $t<0$ gives
    \begin{displaymath}
     \langle h,\xi\rangle_{L^{2}(\Omega)}
     \ge\frac{\E(f+t\xi)-\E(f)}{t}
 \end{displaymath}
 and subsequently, letting $t\to 0-$ leads to
 \begin{equation}
     \label{eq:5}
     \langle h,\xi\rangle_{L^{2}(\Omega)}
     =\langle \E'(f),\xi\rangle_{(W^{1,(2,p)}_{\!\bm{A}, 0})',
   W^{1,(2,p)}_{\!\bm{A}, 0}}. 
 \end{equation}
Since \eqref{eq:5} holds for every $\xi\in
W^{1,(2,p)}_{\!\bm{A},0}(\Omega)$, $h\in L^{2}(\Omega)$ is the
unique extension of the distribution $\E'(f)$ in $L^{2}(\Omega)$.
This completes the proof of this proposition.
\end{proof}

The preceding proposition leads to the following definition.

\begin{defn}\label{def:horizontal-p-laplace}
  For a given Riemannian structure $(\R^d,\bm{A})$ with $\bm{A}$
  satisfying Assumption \ref{ass:1}, we call the $2$nd-order
  differential operator $\Delta_{p}^{\!\bm{A}}$ given by
  \eqref{eq:p-A-Laplacian} the \emph{$p$-Laplace operator} associated
  with $\bm{A}$. Further, for a given open subset $\Omega$ of $\R^{d}$,
  we call the operator
 \begin{displaymath}
   -\Delta_{p}^{\!\bm{A},D}f:=\partial\E(f)
   \qquad\text{for every $f\in D(\partial\E)$,} 
 \end{displaymath}
 where $\E$ denotes the functional defined by~\eqref{eq:3}, the
 \emph{Dirichlet $p$-Laplace operator} in $L^{2}(\Omega)$ associated
 with $\bm{A}$.
\end{defn}

\subsection{Examples of p-Laplace operators}
\label{sec:examples}

As an application of the above given theory, we 
consider the following example. We begin with 
dimension $d=1$.

\begin{exmp}\label{ex41}
  Let $K=\{0\}$, and $\Omega=(a,b)\subseteq \R$ be an open interval for
  given $-\infty<a<0<b<\infty$.  Then, $d_K(x)=\abs{x}$ for every
  $x\in \R$,
  and the matrix $\bm{A}$ has only one coefficient; namely,
\begin{displaymath}
  \bm{A}=\Big[\abs{x}^{\alpha}\Big]
\end{displaymath}
for some given exponent $0<\alpha< 2$. According to
\eqref{eq:def-gradient-in-singular-case}, the associated gradient
\begin{displaymath}
  \nabla_{\!\bm{A}}f(x)=\abs{x}^{\alpha}f'(x):=
  \abs{x}^{\alpha}\frac{\td f}{\dx}(x)
\end{displaymath}
for every $x\in\Omega$ and $f\in C^{1}$. Furthermore, the 
functional $\E$ introduced in~\eqref{eq:3} in \emph{local coordinates}
(see~\eqref{eq:12}) reduces to
\begin{equation}
  \label{bet}
  \E(f)=
  \begin{cases}
  \displaystyle
  \tfrac{1}{p}\int_{a}^{b}\abs{\abs{x}^{\frac{\alpha}{2}}f'(x)}^{p}\,
  \dx & \text{if $f\in W^{1,(2,p)}_{\!\bm{A},0}(a,b)$,}\\[4pt]
  +\infty & \text{otherwise,}
  \end{cases}
\end{equation}
for every $f\in L^2(a,b)$ and according to Proposition~\ref{prop:2},
the \emph{Dirichlet $p$-Laplace operator} $\Delta^{\!\bm{A}}_{p}$ on
$\Omega$ associated with the Riemannian structure
$(\R,\abs{x}^{\beta})$ has the form
\begin{displaymath}
 \Delta_{p}^{\!\bm{A}}f=
 \Big[ \abs{x}^{p\frac{\alpha}{2}}\abs{f'}^{p-2}f'\Big]'
\end{displaymath}
for every $f\in D(\Delta_{p}^{\!\bm{A}})$.
\end{exmp}



%

Next, we consider the higher dimensional case $d=n+m\ge 2$ and the
matrix $\bm{A}$ induces the \emph{generalized Gru\v{s}in} space from
Example \ref{ex:1}~\ref{ex:gen-grusin}.

\begin{exmp}\label{ex:2-gen-grushin}
  For integers $d$, $m$, $n\ge 0$, and for $0\le \alpha<2$, let
  $d=n+m\ge 1$, $\beta_1$, \dots, $\beta_m\ge 0$ and $\bm{A}$ be the
  $d\times d$-diagonal matrix given by
  \eqref{eq:genGrushin-Matrix}. Further, let
  $K:=\{(0,y)\in \R^{n+m}\,\vert\,y\in \R^m\}$ and
  $\Omega\subseteq \R^{d}$ be an open subset such that
  $K\cap \Omega\neq \emptyset$. Then due to the
  relation~\eqref{eq:Grushin-gradient-norm}, the functional
  $\E : L^2(\Omega)\to [0,\infty]$ given by~\eqref{eq:3} reduces to
\begin{displaymath}
  \E(f)=
  \begin{cases}
  \displaystyle
  \tfrac{1}{p}\int_{\Omega}\left[ 
        |x|^{\alpha}_{n}\sum_{i=1}^n\labs{\tfrac{\partial f}{\partial x_i}}^2
        +\sum_{j=1}^{m}\abs{x}_n^{\beta_{j}}
        \labs{\tfrac{\partial f}{\partial y_j}}^2\right]^{\frac{p}{2}}\!
  \dx & \!\!\text{if $f\in W^{1,(2,p)}_{\!\bm{A},0}(\Omega)$,}\\[4pt]
  +\infty & \!\!\text{otherwise,}
  \end{cases} 
\end{displaymath}
for every $f\in L^2(\Omega)$. Accordingly to Proposition~\ref{prop:2},
the \emph{Dirichlet $p$-Laplace operator} $\Delta^{\!\bm{A}}_{p}$ on
$\Omega$ associated with the Riemannian structure $(\R,\bm{A})$ has the
form
\begin{displaymath}
\begin{split}
 \langle\Delta_{p}^{\!\bm{A}}f,g\rangle_{L^2,L^2}
 &=\int_{\Omega}
 \left[|x|^{\alpha}_{n}\sum_{i=1}^n\labs{\tfrac{\partial f}{\partial x_i}}^2
        +\sum_{j=1}^{m}\abs{x}_n^{\beta_{j}}
        \labs{\tfrac{\partial f}{\partial y_j}}^2\right]^{\frac{p-2}{2}}
        \times\\
      &\hspace{2cm} \times\left[|x|^{\alpha}_{n}\sum_{i=1}^{n}
        \tfrac{\partial f}{\partial x_i}\tfrac{\partial g}{\partial x_i} 
        +\sum_{j=1}^{m}\abs{x}_n^{\beta_{j}}
        \tfrac{\partial f}{\partial y_j}
        \tfrac{\partial g}{\partial y_j}\right]\,\dx
 \end{split}
\end{displaymath}
for every $f$, $g\in D(\Delta_{p}^{\!\bm{A}})$.
\end{exmp}

We conclude this section with the following example of a functional $\E$
associated with the sub-Riemannian structure $(\R^{d},\bm{A})$ induced
by \eqref{eq:monomilal} in Example~\ref{ex:1}
\ref{ex:monomial-weight}.

\begin{exmp}
  For integers $d\ge 1$, and for $0\le \alpha_1,\dots, \alpha_d <1$, let
$\bm{A}$ be the $d\times d$-diagonal matrix given by \eqref{eq:monomilal}. Thanks to~\eqref{eq:12}, for the matrix $\bm{A}$ given by \eqref{eq:monomilal}, one has that
\begin{equation}
    \label{eq:ell2-norm-of-A-monomial}
     \abs{\nabla_{\!\bm{A}}f}_{\bm{A}}=\sqrt{\sum_{i=1}^d\labs{\abs{x_i}^{\frac{\alpha_i}{2}}
     \tfrac{\partial f}{\partial x_i}}^2}=\norm{(\abs{x_i}^{\frac{\alpha_i}{2}}\tfrac{\partial f}{\partial x_i})_{i=1}^d}_{\Euc}.
\end{equation}
Thus,  for any open subset $\Omega$ of $\R^d$, the functional $\E : L^2(\Omega)\to [0,+\infty]$ given by~\eqref{eq:3} can be rewritten explicitly as
\begin{displaymath}
  \E(f)=
  \begin{cases}
  \displaystyle
  \tfrac{1}{p}\int_{\Omega}\left[ \sum_{i=1}^d\labs{\abs{x_i}^{\frac{\alpha_i}{2}}
     \tfrac{\partial f}{\partial x_i}}^2
        \right]^{\frac{p}{2}}\!
  \dx & \!\!\text{if $f\in W^{1,(2,p)}_{\!\bm{A},0}(\Omega)$,}\\[4pt]
  +\infty & \!\!\text{otherwise,}
  \end{cases}
\end{displaymath}
for every $f\in L^2(\Omega)$.\medskip 

For later use, we introduce, in addition, the functional
$\tilde \E : C^1_c(\R^d)\to [0,+\infty]$ given by
\begin{displaymath}
  \tilde{\E}(f)=
  \tfrac{1}{p}\int_{\R^d}\abs{\nabla \tilde{f}(x)}^p_{\ell_p}\,\omega(x)\,
  \dx
\end{displaymath}
for every $f\in C^1_c(\R^d)$, where we associate to $f$ a function
$\tilde f$ by setting
\begin{displaymath}
   \tilde{f}(x)= f(x_1^{n_1},\dots,x_d^{n_d})\quad\text{ and }\quad
   \nabla \tilde{f}(x)=\left(\tfrac{\partial\tilde f}{\partial x_1}(x),
     \dots, \tfrac{\partial\tilde f}{\partial x_d}(x)\right)
\end{displaymath}
for every $x=(x_1,\dots,x_d)\in \R^d$ with $n_i:=1/(1-\frac{\alpha_i}{2})$ for
every $i=1,\,\dots, n$, and
\begin{displaymath}
    \omega(x)=|x_1|^{n_1-1}\ldots |x_d|^{n_d-1}.
\end{displaymath}    
Since
\begin{displaymath}
 \frac{\partial \tilde f}{\partial x_i}(x)= 
 \left[\frac{\partial f}{\partial x_i}\right](x_1^{n_1},\dots, x_d^{n_d})
 \,n_i\,x_i^{n_i -1},
\end{displaymath}
applying the substitution $y_i=x_i^{n_i}$ to $\frac{\partial \tilde f}{\partial x_i}$ yields that
\begin{displaymath}
\frac{\partial \tilde f}{\partial x_i}(x) = \left[\frac{\partial f}{\partial x_i}\right](y_1,\dots, y_d)\,n_i\,y_i^{1-\frac{1}{n_i}}=\left[\frac{\partial f}{\partial x_i}\right](y_1,\dots, y_d)\,n_i\,y_i^{\frac{\alpha_i}{2}}
\end{displaymath}
and so,
\begin{align*}
 & \int_{\R^d}\abs{\nabla \tilde{f}(x)}^p_{\ell_p}\,\omega(x)\,\dx\\
  &\qquad = 
    \int_{\R^d}\sum_{i=1}^d\labs{\left[\tfrac{\partial f}{\partial
    x_i}\right](y_1,\dots, y_d)\,n_i\,\abs{y_i}^{\frac{\alpha_i}{2}}}^{p} \,\dy\\
  &\qquad = 
    \int_{\R^d}\sum_{i=1}^d\labs{\left[\tfrac{\partial f}{\partial
    x_i}\right](x)\,n_i\,\abs{x_i}^{\frac{\alpha_i}{2}}}^{p}\,\dx
\end{align*}
(cf. the computation given in~\cite[(3.8) in
  Corollary~3.5]{MR3097258}). Since the norms $\abs{\cdot}_{\Euc}\sim \abs{\cdot}_{\ell_p}$ are
equivalent on $\R^d$, we have thereby shown that
\begin{equation}
  \label{eq:eqiv-E-with-weigths}
  \E(f)\sim \tfrac{1}{p}\int_{\R^d}\abs{\nabla \tilde{f}(x)}^p_{\ell_p}\,\omega(x)\,\dx
\end{equation}
for every $f\in C^{1}_c(\Omega)$.
\end{exmp}


\subsection{The associated carr\'e du champ}
\label{sec:carre-du-champ}
{We should include a discussion about the \emph{carr\'e du champ} $\Gamma$ associated with $-\Delta_{2}^{\bm{A}}$
(see, for example, Section~I.4 of \cite{BH}).}\medskip 

In the case $p=2$ and $\Omega=\R^d$, the operator $\Delta_{2}^{\bm{A}}:=-\partial\E$ in $L^2(\R^d)$ becomes a linear degenerate elliptic operator in divergence form 
\begin{equation}
\label{eq:Delta-A-2}
\Delta^{\bm{A}}_{2} f=\sum_{i,j=1}^{d} \tfrac{\partial}{\partial
    x_j}
\Big(a_{ij}(x) \tfrac{\partial f}{\partial
    x_i}\Big)=\divi(\nabla_{\!\bm{A}}f)\quad\text{in $L^{2}(\R^d)$}
\end{equation}
for every $f\in C^{\infty}_{c}(\R^{d})$. Further, multiplying this equation by $g\in C^{\infty}_{c}(\R^{d})$ with respect to the $L^2$-inner product, and subsequently integrating by parts yields that
\begin{displaymath}
 \langle -\Delta^{\bm{A}}_{2}f,g
 \rangle_{L^2}= \int_{\R^{d}}\langle
 \bm{A}\nabla f,\nabla g\rangle_{\Euc}\,\dx
\end{displaymath}
for every $f$, $g\in C^{\infty}_{c}(\R^{d})$.

Formally, the \emph{carr\'e du champ} $\Gamma$ is given by
\begin{displaymath}
\Gamma_{\!\bm{A}}(f):=f(-\Delta_{2}^{\bm{A}}f)+2^{-1}
\Delta_{2}^{\!\bm{A}}(f^2)
\end{displaymath}
for every $f\in C_c^{\infty}(\R^d)$. Recall that
\begin{displaymath}
  \Gamma_{\!\bm{A}}(f)(x)=\sum_{i,j=1}^{d}a_{ij}(x)\,\tfrac{\partial f}{\partial
    x_i}(x)\,\tfrac{\partial f}{\partial
    x_j}(x)
\end{displaymath}
for every $f\in C_c^{\infty}(\R^d)$ and according to~\cite{MR2778606} and Kato's square root property, one has that
\begin{displaymath}
\norm{\Gamma_{\!\bm{A}}(f)}_1=\tfrac{1}{2}\norm{(-\Delta_{2}^{\bm{A}})^{1/2}f}_2^2
\end{displaymath}
for every $f\in C^{\infty}_c(\R^d)$.

\begin{rem} 
One important task is to understand the intimate relation between the carr\`e du champ operator $\Gamma$ associated with linear $2^{\textrm{nd}}$-order differential operator $\Delta_{2}^{\bm{A}}$ introduced in~\eqref{eq:Delta-A-2} and the nonlinear  \emph{$p$-Laplace operator}
\begin{displaymath}
   \Delta_{p}^{\!\bm{A}}f
   :=\nabla\cdot\left(\abs{\nabla_{\!\bm{A}}f}^{p-2}\nabla_{\!\bm{A}}f\right),
\end{displaymath}
under the Assumption~\ref{ass:1} on the Riemannian structure $(\R^d,\bm{A})$. 
\end{rem}

The last proposition of this section has been 
established in the smooth Riemannian setting in 
\cite[Propostion~2.1]{MR2001444}.

\begin{prop}
  \label{prop:bdd-Riesz-transform}
     Let $(\R^d,\bm{A})$ be a Riemannian structure with matrix $\bm{A}$ satisfying Assumption \ref{ass:1}. Let $1<p<\infty$ and $p^{\mbox{}_{\prime}}=p/(p-1)$. If for some $C_p>0$,
     \begin{equation}
         \label{eq:bddRiesz}
         \norm{\abs{\nabla_{\!\bm{A}}f}_{\bm{A}}}_p \le C_p\,
         \norm{(-\Delta_2^{\!\bm{A}})^{1/2}f}_p\qquad\text{for all }
         f\in C^{\infty}_c(\R^d),
     \end{equation}
     then one has that
     \begin{equation}
         \label{eq:bdd-reverse-Riesz}
         \norm{(-\Delta_2^{\!\bm{A}})^{1/2}f}_{p^{\mbox{}_{\prime}}}
         \le C_p\, \norm{\abs{\nabla_{\!\bm{A}}f}_{\bm{A}}}_{p^{\mbox{}_{\prime}}} 
         \qquad\text{for all }
         f\in C^{\infty}_c(\R^d).
     \end{equation}
\end{prop}

Our proof follows the same idea as given by Coulhon and Duong in
\cite{MR2001444}.

\begin{proof}
    Let $f$, $g\in C^{\infty}_c(\R^d)$. Then, by the definition of $\Delta_{2}^{\!\bm{A}}$ (see~\eqref{eq:Delta-A-2}) and by Green's formula, one has that
    \begin{displaymath}
        \langle -\Delta_{2}^{\!\bm{A}}f,g \rangle_{L^2(\R^d)}=\int_{\R^d}\langle \nabla_{\!\bm{A}}f,\nabla g\rangle_{\Euc}\,\dx
    \end{displaymath}
    For every $m\ge 1$, let $\Omega_m$ be the set given by \eqref{eq:set-Km} where here $\Omega=\R^d$. Then, for every $m\ge 1$, one has that
    \begin{align*}
     \int_{\Omega_m}\langle \nabla_{\!\bm{A}}f,\nabla g\rangle_{\Euc}\,\dx
     &= \int_{\Omega_m}\langle \nabla_{\!\bm{A}}f,\nabla_{\!\bm{A}} g\rangle_{\bm{A}}\,\dx\\
     &\le \int_{\Omega_m}\abs{\nabla_{\!\bm{A}}f}_{\bm{A}}\,\abs{\nabla_{\!\bm{A}}g}_{\bm{A}}\,\dx\\
     &\le \int_{\R^d}\abs{\nabla_{\!\bm{A}}f}_{\bm{A}}\,
     \abs{\nabla_{\!\bm{A}}g}_{\bm{A}}\,\dx
    \end{align*}
     and so, H\"older's inequality yields that
    \begin{displaymath}
        \int_{\Omega_m}\langle \nabla_{\!\bm{A}}f,\nabla g\rangle_{\Euc}\,\dx\le 
        \norm{\abs{\nabla_{\!\bm{A}}f}_{\bm{A}}}_p\,
        \norm{\abs{\nabla_{\!\bm{A}}g}_{\bm{A}}}_{p^{\mbox{}_{\prime}}}
    \end{displaymath}
    for every $m\ge 1$. Since,
    \begin{displaymath}
        \langle \nabla_{\!\bm{A}}f,\nabla g\rangle_{\Euc}\in L^1(\R^d),
    \end{displaymath}
    we have that
    \begin{displaymath}
        \lim_{m\to\infty}\int_{\Omega_m}\langle \nabla_{\!\bm{A}}f,\nabla g\rangle_{\Euc}\,\dx
        =\int_{\R^d}\langle \nabla_{\!\bm{A}}f,\nabla g\rangle_{\Euc}\,\dx
    \end{displaymath}
    and so, we have shown that
    \begin{displaymath}
        \langle -\Delta_{2}^{\!\bm{A}}f,g \rangle_{L^2(\R^d)}
        \le \norm{\abs{\nabla_{\!\bm{A}}f}_{\bm{A}}}_p\,
        \norm{\abs{\nabla_{\!\bm{A}}g}_{\bm{A}}}_{p^{\mbox{}_{\prime}}}.
    \end{displaymath}
    On the other hand, since the square root $(-\Delta_{2}^{\!\bm{A}})^{1/2}$ of $-\Delta_{2}^{\!\bm{A}}$ is self-adjoint, one has that
    \begin{displaymath}
        \langle -\Delta_{2}^{\!\bm{A}}f,g \rangle_{L^2(\R^d)}
        =\langle (-\Delta_{2}^{\!\bm{A}})^{1/2}f,
        (-\Delta_{2}^{\!\bm{A}})^{1/2}g \rangle_{L^2(\R^d)}.
    \end{displaymath}
    Therefore, if \eqref{eq:bddRiesz} holds, then
    \begin{displaymath}
        \langle (-\Delta_{2}^{\!\bm{A}})^{1/2}f,
        (-\Delta_{2}^{\!\bm{A}})^{1/2}g \rangle_{L^2(\R^d)}
        \le C_p\,\norm{(-\Delta_{2}^{\!\bm{A}})^{1/2}f}_p\,
        \norm{\abs{\nabla_{\!\bm{A}}g}_{\bm{A}}}_{p^{\mbox{}_{\prime}}}.
    \end{displaymath}
    By \cite[Lemma~1]{MR1776969}, the set $(-\Delta_{2}^{\!\bm{A}})^{1/2}(C^{\infty}_c(\R^d))$ is dense in $L^{p}(\R^d)$. Hence, 
    \begin{displaymath}
        \norm{(-\Delta_{2}^{\!\bm{A}})^{1/2}g}_{p^{\mbox{}_{\prime}}}
        =\sup_{\stackrel{f\in C^{\infty}_c(\R^d)\text{ s.t. }}{
        \norm{(-\Delta_{2}^{\!\bm{A}})^{1/2}f}_p\le 1}}\left\langle (-\Delta_{2}^{\!\bm{A}})^{1/2}f,
        (-\Delta_{2}^{\!\bm{A}})^{1/2}g \right\rangle_{L^2(\R^d)}
    \end{displaymath}
    and so, the last inequality implies that \eqref{eq:bdd-reverse-Riesz} holds.
\end{proof}


\section{The semigroup generated by the p-Laplace operator}
\label{sec:semigroup-generation}

Due to Proposition \ref{prop:2}, it follows from the classical theory of
nonlinear semigroups in Hilbert spaces (see, for example,
\cite[Th\'eor\`eme~3.2]{MR0348562} or, alternatively,
\cite[Theorem~4.11]{MR2582280}) that for every open subset $\Omega$ of
$\R^d$, the Dirichlet $p$-Laplace operator $\Delta_p^{\!\bm{A},D}$
generates a $C_0$-semigroup $\{e^{t\Delta_p^{\!\bm{A},D}}\}_{t\ge 0}$ of
contractions $e^{t\Delta_p^{\!\bm{A},D}}$ on $L^{2}(\Omega)$ with the
regularizing effect that for every $f\in L^{2}(\Omega)$,
\begin{equation}
  \label{eq:domain-smoothing-effect-semigroup}
  e^{t\Delta_p^{\!\bm{A},D}}f\in D(\Delta_p^{\!\bm{A},D})\qquad\text{for
  every $t>0$.}
\end{equation}
In particular,
we have the following result, where, we write $[u]^{\nu}$ with
$\nu\in \{+,1\}$ for either denoting the \emph{positive part}
$[u]^{+}=\max\{0,u\}$ of $u$ or $[u]^{1}$ for $u$ itself.

\begin{prop}\label{prop:semigroup}
  Let $(\R^d,\bm{A})$ be a Riemannian structure with matrix $\bm{A}$
  satisfying Assumption \ref{ass:1}, and $\Omega$ and open subset
  of $\R^d$. Then, for every $t>0$, the mapping
  $e^{t\Delta_p^{\!\bm{A},D}}$ of the semigroup
  $\{e^{t\Delta_p^{\!\bm{A},D}}\}_{t\ge 0}$ admits a unique contractive
  extension on $L^{q}(\Omega)$ for $1\le q<\infty$ and on
  $\overline{L^{2}(\Omega)\cap
    L^{\infty}(\Omega)}^{\mbox{}_{L^{\infty}}}$
  such that $\{e^{t\Delta_p^{\!\bm{A},D}}\}_{t\ge 0}$ forms a
  $C_0$-semigroup on $L^{q}(\Omega)$ for $1\le q<\infty$. Furthermore,
  for every $f$ and $g\in L^{q}(\Omega)$ if $1\le q<\infty$
  (respectively, $f$ and
  $g\in \overline{L^{2}(\Omega)\cap
    L^{\infty}(\Omega)}^{\mbox{}_{L^{\infty}}}$), one has that
  \begin{displaymath}
       \norm{[e^{t\Delta_p^{\!\bm{A}}}f-e^{t\Delta_p^{\!\bm{A}}}g]^{\nu}}_{q}
       \le \norm{[e^{s\Delta_p^{\!\bm{A}}}f-e^{s\Delta_p^{\!\bm{A}}}g]^{\nu}}_{q} 
  \end{displaymath}
 for every $0\le s<t<\infty$, and $\nu\in \{+,1\}$.
\end{prop}

\begin{proof}
  Accordingly to \cite{MR1164641} (see also \cite{CoulHau2017}), the
  statement of this proposition holds provided the $p$-Laplace operator
  $\Delta_p^{\!\bm{A}}$ satisfies
  \begin{equation}
    \label{eq:complete-accretive}
     \int_{\Omega}p(f-g)(v-\hat{v})\,\dx\ge 0
   \end{equation}
   for every $p\in P_{0}$ and every $(f,v)$,
   $(g,\hat{v})\in \Delta_p^{\!\bm{A}}$, where $P_{0}$ is the set of
   (smooth) truncator functions given by
   \begin{displaymath}
     P_{0}=\Big\{ p\in C^{\infty}(\R)\,\Big\vert\;0\le p'\le
     1,\;\supp(p')\text{ compact, }0\notin \supp(p)\,\Big\}. 
   \end{displaymath}
   To see that \eqref{eq:complete-accretive} holds, let $u$,
   $\hat{u}\in D(\Delta_p^{\!\bm{A}})$ and $p\in P_0$. Then,
   \allowdisplaybreaks
   \begin{align*}
       &\int_{\Omega}p(f-g)(\Delta_p^{\!\bm{A}}f-\Delta_p^{\!\bm{A}}g)\,\dx\\
       &\quad =
       \int_{\Omega}
      \langle\abs{\nabla_{\!\bm{B}}f}^{p-2}_{\Euc}
      \nabla_{\!\bm{B}}f
      -\abs{\nabla_{\!\bm{B}}g}^{p-2}_{\Euc}
      \nabla_{\!\bm{B}}g,\bm{B}\nabla p(f-g)\rangle_{\Euc}
       \,\dx\\
     &\quad = 
      \int_{\Omega}
      \langle
      \abs{\nabla_{\!\bm{B}}f}^{p-2}_{\Euc}
      \nabla_{\!\bm{B}}f
      -\abs{\nabla_{\!\bm{B}}g}^{p-2}_{\Euc}
      \nabla_{\!\bm{B}}g,\bm{B}\nabla(f-g)\rangle_{\Euc}\, p'(f-g)
       \,\dx\\
       &\quad \ge0,
   \end{align*}
   since $p'\ge 0$ and by the convexity of $x\mapsto \abs{x}^p$, one has that
   \begin{align*}
       &\langle
      \abs{\nabla_{\!\bm{B}}f}^{p-2}_{\Euc}
      \nabla_{\!\bm{B}}f
      -\abs{\nabla_{\!\bm{B}}g}^{p-2}_{\Euc}
      \nabla_{\!\bm{B}}g,\bm{B}\nabla(f-g)\rangle_{\Euc}\\
      &\qquad=
      \langle
      \abs{\bm{B}\nabla f}^{p-2}_{\Euc}
      \bm{B}\nabla f
      -\abs{\bm{B}\nabla g}^{p-2}_{\Euc}
      \bm{B}\nabla g,\bm{B}\nabla f-\bm{B}\nabla g)\rangle_{\Euc}\ge 0
   \end{align*}
   a.e. on $\R^d$. An alternative proof of the statement of this
   proposition can be obtained by verifying the condition on the
   function $\E$ given in \cite[Theorem 3.9]{MR3465809} (see also
   \cite[Theorem 3.10]{MR3465809}).
\end{proof}

It is important to mention that the semigroup
$\{e^{t\Delta_p^{\!\bm{A}}}\}_{t\ge 0}$ admits further regularization
properties; for instance, for every $f\in L^q(\Omega)$,
$1\le q\le \infty$, the time-derivative
$\tfrac{\td}{\dt}e^{t\Delta_p^{\!\bm{A}}}f$ global
$L^q(\Omega)$-estimates implying an immediate regularization effect and
decay in time (cf. \cite{MR648452}, and see also
also~\cite{MR4031770,MR4200826}). But for the aims of this paper, these
estimates are not relevant.


\section{Nash and Sobolev inequalities in Gru\v{s}in 
and spaces with monomial weights}
\label{sec:Lq-Lr-regularization_effect}

Throughout this section, we focus on the Gru\v{s}in space setting
$(\R^d,\bm{A})$ from Example~\ref{ex:1}~\ref{ex:gen-grusin} and the
structure $(\R^d,\bm{A})$ inducing spaces with monomial weights from
Example~\ref{ex:1}~\ref{ex:monomial-weight}. 

\subsection{Sub-elliptic estimates in Gru\v{s}in spaces}\label{subsec:sub-elliptic-Grushin}

For integers $d$, $m$, $n\ge 0$, and for $0\le \alpha<2$, let
$d=n+m\ge 1$, $\beta_1$, \dots, $\beta_m\ge 0$, and $\bm{A}$ be the
$d\times d$-diagonal matrix given by
\eqref{eq:genGrushin-Matrix}.\medskip

We begin by deriving a Nash inequality associated with the degenerate
matrix $\bm{A}$. Thereby, we follow some ideas from \cite{RSik}. For the
Nash inequality associated with the generalized Gru\v{s}in matrix
$\bm{A}$ given by \eqref{eq:genGrushin-Matrix}, the new dimension
\begin{equation}
    \label{eq:dimension-Grushin}
    D:=\left(n+ m(1-\alpha/2) + \frac{1}{2}\sum_{k=1}^m \beta_k \right)(1-\alpha/2)^{-1} 
\end{equation}
is crucial. Note that $D\ge n+m$.\medskip 

\begin{thm}
    \label{thm:Grushin-Nash-inequality}
    For integers $d$, $m$, $n\ge 0$, and for $0\le \alpha<2$, let
    $d=n+m\ge 1$, $\beta_1$, \dots, $\beta_m\ge 0$, $D$ be given by
    \eqref{eq:dimension-Grushin}, and $\bm{A}$ the generalized
    Gru\v{s}in matrix given by \eqref{eq:genGrushin-Matrix}. Then, there
    is a $C>0$ such that the following Nash-inequality
    \begin{equation}
       \label{eq:Grushin-Nash-inequality}
        \norm{f}_2\le C\,        
        \norm{\abs{\nabla_{\!\bm{A}}f}_{\bm{A}}}_2^{\frac{D}{D+2}}\,
        \norm{f}_{1}^{1-\frac{D}{D+2}}
    \end{equation}
    holds for every $f\in C_{c}^{1}(\R^d)$.
\end{thm}

For the proof of our (degenerate) Nash inequality, we need to introduce
some further notions. Given a positive real-valued function $Q$ on
$\R^d$, let $Q(i\nabla_{\Euc})$ be the $Q$ the positive self-adjoint
operator on $L^2(\R^d)$ defined by the theory of Fourier multipliers
$Q(i\nabla_{\Euc})=\mathcal{F}^{-1}Q(\cdot)\mathcal{F}$. Here, we denote
by $\mathcal{F}[\varphi]=\hat{\varphi}$ the Fourier transform of a
function $\varphi\in L^2(\R^d)$, and by $\mathcal{F}^{-1}$ its
inverse. Further, let $q : D(q)\times D(q)\to \R$ denote the closed
linear form associated with $Q$; that is,
\begin{equation}\label{qq}
    q(\varphi):=\int_{\R^d}Q(\xi)\,
    \abs{\hat{\varphi}(\xi)}^2\,\td\xi
\end{equation}
for every $\varphi\in D(q)$, where the domain $D(q)$ of $q$ is linear
subspace of $L^2(\R^d)$ such that the integral of $q$ is
finite. Finally, for given $Q$ and $r>0$, let
\begin{displaymath}
    V_{Q}(r):=\labs{\Big\{\xi\in\R^{n+m}\,\Big\vert\,Q(\xi)<r^2\Big\}}.
\end{displaymath}
Recall that 
c\begin{lem}\label{lem6.2} 
If $h\ge q$ then 
 $$\norm{f}_2^2\le C \left(r^{-2} h(f) + V_Q(r)\norm{f}_1^2   \right)
 =C\left(r^{-2}\norm{\abs{\nabla_{\!\bm{A}}f}_{\bm{A}}}_2^2+ V_Q(r)\norm{f}_1^2   \right).
 $$
\end{lem}
For the proof and discussion of the above version of Nash inequality,
see e.g. \cite[Lemma 2.2]{RSik}.

With this in mind, we can now outline the proof of the Nash
inequality~\eqref{eq:Grushin-Nash-inequality}.

\begin{proof}[Proof of Theorem~\ref{thm:Grushin-Nash-inequality}]
  The proof is adapted from \cite{RSik} and requires only small minor
  changes to the argument described there.  We consider the form
  described in \eqref{eq:Grushin-gradient-norm} and set
\begin{equation*}
  h(f) = \abs{\nabla_{\!\bm{A}}f}_{\bm{A}}^2
  =|x|^{\alpha}_{n}\sum_{i=1}^n\labs{\tfrac{\partial f}{\partial x_i}}^2
  +\sum_{j=1}^{m}\abs{x}_n^{\beta_{j}}\labs{\tfrac{\partial f}{\partial y_j}}^2.
     \end{equation*}

     By \cite[Lemma 2.3]{RSik} it is enough to find $Q$ on $\R^d$, such
     that $h \ge q$ and $V_Q \le c r^D$, where $c$ is a positive
     constant and $q$ is a quadratic form associated with $Q$ in a way
     described in \eqref{qq}. We represent $h$ as
     $h=h_0 +\sum_{j=1}^{m} h_j $ where we set
     $$
     h_0(f)=|x|^{\alpha}_{n}\sum_{i=1}^n\labs{\tfrac{\partial f}{\partial x_i}}^2  \quad \mbox{and} 
     \quad h_j(f)= \abs{x}_n^{\beta_{j}}\labs{\tfrac{\partial f}{\partial y_j}}^2. $$

     It was verified in \cite[Proposition 3.5]{RSik} that for some
     constant $c>0$ any $1 \le j \le m$
$$
h_0+h_j \ge cL_{y_j}^{\gamma_j} \quad \mbox{and} 
     \quad  h_0 \ge c L_{x}^{(1-\alpha/2)}
$$
where  $\gamma_j= \frac{1-\alpha/2}{1+(\beta_j -\alpha)/2 }$,  $L_x =  -\sum_{i=1}^n\partial_{x_i}^2 $,  and $L_y= -\partial^2_{y_j}$.
It follows that for some $c>0$
$$
c^{-1}h \ge  L_{x}^{(1-\alpha/2)}+ \sum_{j=1}^{m} L_{y_j}^{\gamma_j}. 
$$
That is $h \ge CQ$ where $Q$ is defined by
$$
Q(\xi_x,\xi_y)=|\xi_x|^{2(1-\alpha/2)} +\sum_{j=1}^{m} |\xi_{_j}|^{2\gamma_j}.
$$
Set $\sigma = \sum_{j=1}^{m} \gamma_j^{-1} =r^{D}$ and  note that 
$$
 V_{Q}(r):=\labs{\Big\{\xi\in\R^d\,\Big\vert\,Q(\xi)<r^2\Big\}} \sim r^{n/(1-\alpha)} r^\sigma =r^{D}
$$
Now Theorem~\ref{thm:Grushin-Nash-inequality} follows by minimizing with respect to $r$ the inequality stated in Lemma \ref{lem6.2}.
\end{proof}

According to Proposition~\ref{prop:semigroup}, the linear operator
$-\Delta_2^{\!\bm{A}}$ introduced in~\eqref{eq:Delta-A-2} generates a
linear $C_0$-semigroup $\{e^{-t\Delta_2^{\!\bm{A}}}\}_{t\ge 0}$ acting
on $L^q(\R^d)$ for all $1\le q\le \infty$.\medskip

Using the argument described in in \cite{RSik} we can show that it
follows from Theorem~\ref{thm:Grushin-Nash-inequality} and Lemma
\ref{lem6.2} that
\begin{displaymath}
 \norm{e^{-t \Delta_2^{\!\bm{A}}}}_{1\to \infty} \lesssim\,t^{-D/2}
\end{displaymath}
see \cite[Lemma 2.5]{RSik}. By \cite[Theorem~II.4.3 \&
Remark~II.4.4]{MR1218884} (see also \cite{MR2396848}), the previous
$L^1$-$L^{\infty}$ regularization inequality is equivalent to
\begin{displaymath}
    \norm{(\Delta_2^{\!\bm{A}})^{-1/2}f}_{r} \le C\,\norm{f}_p
\end{displaymath}
for every $f\in C^{\infty}_c(\R^d)$ and $1<p<r<\infty$ provided
$1/p-1/r=1/D$. Therefore, for such $p$ and $r$, one has that
\begin{equation}
   \label{eq:Sobolev-for-fractional-Delta-A}
   \norm{f}_r \lesssim\, \norm{(\Delta_2^{\!\bm{A}})^{1/2}f}_p
\end{equation}
for every $f\in C^{\infty}_c(\R^d)$.  In \cite[Section 8.1]{RSik}, it
was proved that the \emph{Riesz transform}
$R:=\nabla (\Delta_2^{\!\bm{A}})^{-1/2}$ is bounded on $L^{q}(\R^d)$ for
all $1<q\le 2$. Therefore, for every $1<q\le 2$, there is a $C_q>0$ such
that
\begin{displaymath}
 \norm{|\nabla_{\!\bm{A}} f|_{\bm{A}}}_{q} \le\,C_q\, 
   \norm{(\Delta_2^{\!\bm{A}})^{1/2}f}_{q}   
\end{displaymath}
for every $f\in C^{\infty}_c(\R^d)$. Now, by
Proposition~\ref{prop:bdd-Riesz-transform}, 
for every $q'\ge 2$, there is a constant
$C_{q'}>0$ such that
\begin{equation}
    \label{eq:inverse-Riesz}
    \norm{(\Delta_2^{\!\bm{A}})^{1/2}f}_{q'} \le\,C_{q'}\, 
    \norm{\abs{\nabla_{\!\bm{A}}f}_{\bm{A}}}_{q'}
\end{equation}
for every $f\in C^{\infty}_c(\R^d)$. If we now
apply~\eqref{eq:inverse-Riesz} to
\eqref{eq:Sobolev-for-fractional-Delta-A} for
$p=q'\ge 2$ then we obtain the following \emph{Sobolev
  inequality} for the degenerate matrix $\bm{A}$ given
by~\eqref{eq:genGrushin-Matrix}.

\begin{thm}\label{thm:Sobolev-genGrusin}
  For integers $d$, $m$, $n\ge 0$, and for $0\le \alpha<2$, let
  $d=n+m\ge 1$, $\beta_1$, \dots, $\beta_m\ge 0$, $D$ be given by
  \eqref{eq:dimension-Grushin}, and $\bm{A}$ the generalized Gru\v{s}in
  matrix given by \eqref{eq:genGrushin-Matrix}. Then, for $2\le p<D$,
  one has that
 \begin{equation}
    \label{Sobolev:genGrusin}
    \norm{f}_{\frac{pD}{D-p}}\lesssim\, \norm{\abs{\nabla_{\!\bm{A}}f}_{\bm{A}}}_{p}
\end{equation}
for every $f\in C^{\infty}_c(\R^d)$. 
\end{thm}

\begin{rem}
  a.) It is worth noting that Robinson and the second author considered
  the operator $ \Delta_2= \Delta_{\R^n} + |x|^{2N}\Delta_{\R^m}$
  corresponding to the matrix $A$ with $\alpha=0$ and $\beta_j=2N$ for
  some $N\in \N$. b.) In \cite{Rsik2} they established boundedness of
  the corresponding Riesz Transform for all $1<p< \infty$. This implies
  Sobolev inequality~\eqref{Sobolev:genGrusin} for all
  $1<p<D$. c.) Furthermore, an alternative proof of
  Theorem~\ref{thm:Sobolev-genGrusin} avoiding the boundedness property
  of the Riesz transform $R$ can be constructed by using the theory
  developed in \cite{MR1386760}.
\end{rem}


\subsection{Sobolev inequality in spaces of monomial weights}\label{subsec:sub-elliptic-Grushin-monomial}

Next, we recall the following Sobolev
inequality due to Cabr\'e \& Ros-Oton~\cite[Corollary~3.5]{MR3097258}.

\begin{thm}
 \label{thm:Sobolev-genGrusin-monomial}
  For $d\in \N$ and for $\alpha_{1}$, \dots, $\alpha_d\in [0,2)$, let
  \begin{equation}
   \label{eq:dimension-Grushin-monomial}
      D=d+\sum_{i=1}^{d}\frac{\alpha_i}{2-\alpha_i}
  \end{equation}
  and $1\le p<D$. Further, let $(\R^d,\bm{A})$
  be the Riemannian structure given by \eqref{eq:monomilal}. Then,
    \begin{equation}
    \label{Sobolev:genGrusin-monomial}
    \norm{f}_{\frac{pD}{D-p}}\lesssim\, \norm{\abs{\nabla_{\!\bm{A}}f}_{\bm{A}}}_{p}
\end{equation}
for every $f\in C^{\infty}_c(\R^d)$. 
\end{thm}

Before we outline the proof of this theorem, we want to fix two comments.

\begin{rem}
  a.) For $2\le p<D$, one would obtain the Sobolev
  inequality~\eqref{Sobolev:genGrusin-monomial} by using the the
  methodology for proving Theorem~\ref{thm:Sobolev-genGrusin}. b.) In
  the case $m=0$, the dimension $D$ introduced in
  \eqref{eq:dimension-Grushin} coincides with $D$ in
  \eqref{eq:dimension-Grushin-monomial}. The proof is essentially the
  same as in Lemma~\ref{lem6.2}.
\end{rem}

\begin{proof}
  Due to the relation \eqref{eq:eqiv-E-with-weigths}, the Sobolev
  inequality \eqref{Sobolev:genGrusin} follows from the weighted Sobolev
  inequality \cite[Theorem~1.3]{MR3097258} with implying
  $A_i:=n_i=1/(1-\alpha_i/2)$. Alternatively, we could used the
  equivalence between the $\ell_p$- and $\Euc$-norm on $\R^d$, and
  subsequently apply the Sobolev inequality~\cite[(3.8) in
  Corollary~3.5]{MR3097258} with $\tilde{\alpha}_i=\alpha_i/2$,.
\end{proof}


\section{The regularizing effects of the semigroups}
\label{subsec:LqLr-regularization}

Due to \cite[Theorem~1.2 \& Theorem~1.4]{CoulHau2017}, we can conclude
the following $L^r$-$L^\infty$ regularization effect of the semigroup
$\{e^{-t\Delta_{p}^{\!\bm{A},D}} \}_{t\ge 0}$ generated by the Dirichlet
$p$-Laplace operator $\Delta_{p}^{\!\bm{A},D}$ from the Sobolev
inequalities~\eqref{Sobolev:genGrusin} and
\eqref{Sobolev:genGrusin-monomial} in the previous section. Throughout
this section, $\Omega$ denotes an open subset of $\R^d$.

\subsection{The regularizing effects of the semigroup in Gru\v{s}in spaces}
\label{subsec:regularity-in Grushin} 
First, we outline the regularization effects of the semigroup
$\{e^{-t\Delta_{p}^{\!\bm{A},D}} \}_{t\ge 0}$ generated by the Dirichlet
$p$-Laplace operator $\Delta_{p}^{\!\bm{A},D}$, when the
Sub-Rie\-mannian structure $(\R^d,\bm{A})$ induces the generalized
Gru\v{s}in space.

 \begin{thm}\label{thm65}
   For integers $d$, $m$, $n\ge 0$, and for $0\le \alpha<2$, let
   $d=n+m\ge 1$, $\beta_1$, \dots, $\beta_m\ge 0$, and $\bm{A}$ be given
   by \eqref{eq:genGrushin-Matrix}. Further, suppose $D$ is given by
   \eqref{eq:dimension-Grushin} and $2\le p<D$. Then, for every
   $1\le q\le \frac{Dp}{D-p}$, the semigroup
   $\{e^{-t\Delta_{p}^{\!\bm{A}}}\}_{t\ge 0}$ satisfies
	\begin{equation}
        \label{eq:L1-Linfty}
    	 \norm{e^{-t\Delta_{p}^{\!\bm{A}}}f-e^{-t\Delta_{p}^{\!
      \bm{A}}}g}_{\infty} \lesssim t^{-\delta_r}\,
      \norm{f-g}_q^{\gamma_q}    
	\end{equation}
	for every $t>0$, $f$, $g\in L^q(\Omega)$ with the exponents
    \begin{equation}
      \label{eq:exp-part1}
       \delta_{q}=\frac{\delta^{\ast}}{1-\gamma^{\ast}\left(1-\frac{q(D-p)}{Dp} \right)},\qquad
        \gamma_{q}=\frac{\gamma^{\ast}\,q(D-p)}{Dp
        \left(1-\gamma^{\ast}\left(1-\frac{q (D-p)}{Dp}
        \right)\right)},
    \end{equation}
    where
    \begin{equation}
      \label{eq:exp-part2}
     \delta^{\ast}=\frac{D-p}{p^2+(D-p)(p-2)},\qquad
      \gamma^{\ast}=\frac{p^2}{p^2+(D-p)(p-2)}.
    \end{equation}
\end{thm}

\begin{rem}
  It worth noting that the $L^1$-$L^\infty$ regularity
  estimates~\eqref{eq:L1-Linfty} obtained here are consistent with the
  case $p=2$ and partially generalize the ones obtained in \cite[Lemma
  3.2 and Proposition 3.1]{RSik}. It would be interesting to know the
  $L^1$-$L^\infty$ regularity estimates~\eqref{eq:L1-Linfty} in the case
  $1<p<2$ and $p\ge D$.
\end{rem}

\subsection{The regularizing effects of the semigroup in spaces with
  monomial weights}
\label{subsec:regularity-in-spaces-with-monomial-weights} 
Next, we state the regularization effects of the semigroup
$\{e^{-t\Delta_{p}^{\!\bm{A},D}}\}_{t\ge 0}$ generated by the Dirichlet
$p$-Laplace operator $\Delta_{p}^{\!\bm{A},D}$, when the Sub-Riemannian
structure $(\R^d,\bm{A})$ induces spaces with monomial weights.

 \begin{thm}\label{thm66-monomial}
   For $d\in \N$ and for $\alpha_{1}$, \dots, $\alpha_d\in [0,2)$, 
   let $\bm{A}$ be the matrix from~\eqref{eq:monomilal}, and $D$
   be given by \eqref{eq:dimension-Grushin-monomial}. Then, the
   following statements hold.
   \begin{enumerate}
    \item If $1<p\le 2D/(D+2)$, then for every $q_{0}\ge p$ satisfying
        \begin{displaymath}
          (D/(D-p)-1)\,q_{0}+p-2>0,
        \end{displaymath}
        and every $1\le q\le D\, q_{0}/(D-p)$ satisfying
        $q>D(2-p)/p$, the semigroup $\{e^{-t\Delta_{p}^{\!\bm{A},D}}\}_{t\ge 0}$ satisfies
        \begin{equation}
          \label{eq:Lq-Linfty-estimates-diff}
          \norm{T_{t}u}_{\infty}\lesssim\,t^{-\delta_{q}} \,\norm{u}_{q}^{\gamma_{q}}
        \end{equation}
        for every $t>0$, $u$, $\hat{u}\in L^{q}(\Sigma)$, with exponents
        \begin{equation}
          \label{eq:exponents-dirichlet-estimates}
           \delta_{q}=\frac{\delta^{\ast}_{q_0}}{1-\gamma^{\ast}\left(1-\frac{q(D-p)}{D
                  q_{0}}\right)},\quad
            \gamma_{q}=\frac{\gamma^{\ast}_{q_0}\,q(D-p)}{D
              q_{0}\left(1-\gamma^{\ast}\left(1-\frac{q (D-p)}{D
                    q_{0}}\right)\right)},
        \end{equation}
        with
        \begin{displaymath}
          \delta^{\ast}_{q_0}=\frac{D-p}{p\, q_{0}+(D-p)(p-2)},\quad
          \gamma^{\ast}_{q_0}=\frac{p\, q_{0}}{p\,
              q_{0}+(D-p)(p-2)}.
        \end{displaymath}
        
      \item If $2D/(D+2)<p\le 2D/(D+1)$, then for every
        $1\le q\le \frac{D p}{D-p}$ satisfying $q>\frac{D(2-p)}{p}$, the
        semigroup $\{e^{-t\Delta_{p}^{\!\bm{A},D}}\}_{t\ge 0}$ satisfies
        the $L^{q}$-$L^{\infty}$-regularizing
        estimate~\eqref{eq:Lq-Linfty-estimates-diff} with the exponents
        $\delta_{q}$, $\gamma_{q}$ given by \eqref{eq:exp-part1} and
        \eqref{eq:exp-part2}.

        \item If $\frac{2d}{d+1}<p<2$, then for every $1\le q\le \frac{D
            p}{D-p}$, the semigroup $\{e^{-t\Delta_{p}^{\!\bm{A},D}}\}_{t\ge 0}$
        satisfies the $L^{q}$-$L^{\infty}$-regularization
        estimate~\eqref{eq:Lq-Linfty-estimates-diff} with the exponents
        $\delta_{q}$, $\gamma_{q}$ given by \eqref{eq:exp-part1} and
        \eqref{eq:exp-part2}.
 
    \item For every $2\le p<D$ and $1\le q\le \frac{Dp}{D-p}$, the semigroup
     $\{e^{-t\Delta_{p}^{\!\bm{A},D}}\}_{t\ge 0}$ satisfies the
     $L^q$-$L^{\infty}$-regularizing estimate~\eqref{eq:L1-Linfty} for
     $1\le q\le \frac{Dp}{D-p}$  with the exponents $\delta_{q}$,
     $\gamma_{q}$ given by \eqref{eq:exp-part1} and \eqref{eq:exp-part2}.
   \end{enumerate}
\end{thm}

We conclude this section with the following remark.

\begin{rem}
  In addition, to the preceding theorem, we could also apply the
  Trudinger-Moser inequality (for $p=D$) and the Morrey inequality (for
  $p>D$) from \cite{MR3097258} to derive the
  $L^{q}$-$L^{\infty}$-regularizing
  estimate~\eqref{eq:Lq-Linfty-estimates-diff} of the semigroup
  semigroup $\{e^{-t\Delta_{p}^{\!\bm{A},D}}\}_{t\ge 0}$ on bounded subsets
  $\Omega$ of $\R^d$.
\end{rem}


\section{Separation phenomena: Dichotomy of continuity}
\label{sec:Separation}

This section is dedicated to the separation phenomena. For this, we focus
throughout this section on the sub-Riemannian structure $(\R^d,\bm{A})$
given by \eqref{eq:genGrushin-Matrix} for $d=1+m$, $0\le \alpha<2$ and
$\beta_1$, \dots, $\beta_m \ge0$. We emphasize that the same separation
phenomena holds for the sub-Riemannian structure $(\R^d,\bm{A})$
given by \eqref{eq:dimension-Grushin-monomial} and at several
hyperplanes $\{x_1=0\}$, \dots, $\{x_n=0\}$.\medskip

More precisely, let
\begin{equation}
  \label{eq:genGrushin-Matrix-simplified}
       \bm{A}=\begin{bmatrix} 
       |x_1|^{\alpha}\, & 0 \\
       0 & \begin{bmatrix}
       |x_1|^{\beta_{1}} & 0 & \cdots & 0\\
       0& |x_1|^{\beta_{2}} & \ddots &  0\\
       \vdots & \ddots & \ddots & 0\\
       0 & \cdots &   0 & |x_1|^{\beta_{m}}
       \end{bmatrix}_{m\times m}
  \end{bmatrix}.
\end{equation}
Then the length of the gradient $\nabla_{\!\bm{A}}f$ in local coordinates 
\begin{equation}
        \label{eq:Grushin-gradient-norm-simplified}
        \abs{\nabla_{\!\bm{A}}f}_{\bm{A}}^2=|x_1|^{\alpha}\labs{\tfrac{\partial f}{\partial x_1}}^2+
        \sum_{j=1}^{m}|x_1|^{\beta_j}\labs{\tfrac{\partial f}{\partial y_j}}^2.
\end{equation}
and the associated Gru\v{s}in-type $p$-Laplace operator $\Delta_{p}^{\!\bm{A}}$ on an open subset $\Omega$ of $\R^{d+1}$ is given by
\begin{displaymath}
\begin{split}
 \langle\Delta_{p}^{\!\bm{A}}f,\xi\rangle_{\D'(\Omega),\D(\Omega)}
 &=\int_{\Omega}
 \left[|x_1|^{\alpha}\,\labs{\tfrac{\partial f}{\partial x_1}}^2
        +\sum_{j=1}^{m}\abs{x_1}^{\beta_{j}}\labs{\tfrac{\partial f}{\partial y_j}}^2\right]^{\frac{p-2}{2}}
        \times\\
      &\hspace{2cm} \times\left[|x_1|^{\alpha}
      \tfrac{\partial f}{\partial x_1}\tfrac{\partial \xi}{\partial x_1} 
        +\sum_{j=1}^{m}\abs{x_1}^{\beta_{j}}\tfrac{\partial f}{\partial y_j}\tfrac{\partial \xi}{\partial y_j}\right]\,\dx
 \end{split}
\end{displaymath}
for every $f\in W^{1,p}_{\!\bm{A},loc}(\Omega)$ and $\xi\in
C^{\infty}_c(\Omega)$. Since the separation phenomena is established
with respect to the space variable, we first to the stationary equation.

\subsection{The elliptic  equation}
\label{sec:separation-elliptic}

We begin this subsection by introducing the following notion of solutions. 

\begin{defn}\label{def1}
    Let $\Omega\subseteq \R^{d}$ be an open subset. Then, a function $f\in W^{1,p}_{\!\bm{A},loc}(\Omega)$ is called a \emph{weak solution} of the Gru\v{s}in-type $p$-Laplace equation
    \begin{equation}
    \label{eq:13}
      -\Delta_{p}^{\! \bm{A}}f=0
      \qquad\text{in $\Omega$,} 
    \end{equation}
    if
    \begin{displaymath}
     \langle\Delta_{p}^{\!\bm{A}}f,g\rangle_{\D'(\Omega),\D(\Omega)} =0
    \end{displaymath}
    for every $\xi\in C^{\infty}_{c}(\Omega)$.
\end{defn}

\begin{rem}\label{rem:Ap-weights}
  It is worth mentioning that in the case
  $\beta_1=\cdots =\beta_m=\alpha$ and $1<p<\infty$, the weight
  $\abs{x_1}^{\frac{\alpha}{2}p}$ belongs to the $A_{p}$ Muckenhoupt
  class if and only if $-1<\frac{\alpha}{2}p<p-1$
  (see~\cite[p. 10]{MR1207810}). Note, this condition on $\alpha$ is
  equivalent to
  $-\frac{1}{p}<\alpha<\frac{2}{p^{\mbox{}_{\prime}}}=2(p-1)/p$. According
  to \cite[Theorem~3.70]{MR1207810}, if
  $-\frac{1}{p}<\alpha<\frac{2}{p^{\mbox{}_{\prime}}}$, then every weak
  solution $f$ of~\eqref{eq:13} admits a continuous representative
  $\bar{f}$ such that $f=\bar{f}$ a.e. on $\Omega$.
\end{rem}

\begin{defn}
  Let $\Omega\subseteq \R^d$ be an open subset. Then, a function
  $f : \Omega\to \R$ is called to be \emph{$(\bm{A},p)$-harmonic} on
  $\Omega$ if $f$ is a continuous weak solution of the Gru\v{s}in-type
  $p$-Laplace equation~\eqref{eq:13}.
\end{defn}

\begin{rem}\label{rem:harnack-Ap-weights}
  Further, in the case $\beta_1=\cdots =\beta_m=\alpha$ and
  $1<p<\infty$, it follows from~\cite[Theorem 6.2]{MR1207810} that for
  the range $-\frac{1}{p}<\alpha<\frac{2}{p^{\mbox{}_{\prime}}}$, every
  positive $(\bm{A},p)$-harmonic function defined on an open subset
  $\Omega\subset \R^d$ (which might include the hyper-plane $\{x_1=0\}$)
  satisfies a \emph{Harnack inequality}. Hence, in particular, every
  $(\bm{A},p)$-harmonic function defined on a domain
  $\Omega\subset \R^d$ is H\"older continuous.
\end{rem}

The next theorem shows that if
$\frac{2}{p^{\mbox{}_{\prime}}}\le \alpha<2$, then every
$(\bm{A},p)$-harmonic function defined on an open subset
$\Omega\subset \R^d$ containing the hyperplane $\{x_1=0\}$ does not need to
be continuous at $\{x_1=0\}$. Thus the following theorem proves an
optimality condition of the class of $A_{p}$ Muckenhoupt weights for
obtaining continuity of weak solutions $f$ of equation~\eqref{eq:13}.

\begin{thm}
  \label{thm:5}
  Let $1<p<\infty$, $p^{\mbox{}_{\prime}}=p/(p-1)$, and for
  $0\le \alpha<2$, $\beta_1$, \dots, $\beta_m \ge 0$, let
  $(\R^d,\bm{A})$ be the Riemannian structure given by \eqref
  {eq:genGrushin-Matrix-simplified}. Further, let $\Omega\subseteq \R^d$
  be an open subset such that $\Omega\cap \{x_1=0\}\neq \emptyset$. If
  $\frac{2}{p^{\mbox{}_{\prime}}}\le \alpha<2$, then a weak solution
  $f \in W^{1,p}_{\! \bm{A},loc}(\Omega)$ of the Gru\v{s}in-type
  $p$-Laplace equation~\eqref{eq:13} does not need to be continuous
  along $y_j$ crossing $\{x_1=0\}$ for every $j=1$, \dots, $m$.
\end{thm}



For the proof of Theorem~\ref{thm:5} the following lemma is crucial.
For a given open subset $\Omega\subseteq \R^d$ $d=1+m$ satisfying
$\Omega\cap \{x_1=0\}\neq \emptyset$, we write $\Omega^{+}$ for
$\Omega\cap \{x_1>0\}$ and set $\Omega^{-}=\Omega\cap\{x_1<0\}$.
Further, we note that for the Riemannian structure $(\R^d,\bm{A})$ given
by \eqref{eq:genGrushin-Matrix-simplified} for an open subset
$\Omega\subseteq \R^{d}$, the functional $\E$ associated with the
Gru\v{s}in-type $p$-Laplace operator $\Delta_{p}^{\!\bm{A}}$ is given by
\begin{equation}
 \label{eq:Grushin-p-energy-simply}
  \E(f)=
  \tfrac{1}{p}\int_{\Omega}\left[ 
        |x_1|^{\alpha}\labs{\tfrac{\partial f}{\partial x_1}}^2
        +\sum_{j=1}^{m}\abs{x_1}^{\beta_{j}}\labs{\tfrac{\partial f}{\partial y_j}}^2\right]^{\frac{p}{2}}\!
  \dx 
\end{equation}
for every $f\in L^1_{loc}(\Omega)$ with $\abs{\nabla_{\!\bm{A}}f}_{\bm{A}}\in L^p(\Omega)$.

\begin{lem}
  \label{lem:5}
  Let $1<p<\infty$, $p^{\mbox{}_{\prime}}=p/(p-1)$, and for
  $0\le \alpha<2$, $\beta_1$, \dots, $\beta_m \ge 0$, let
  $(\R^d,\bm{A})$ be the Riemannian structure given by \eqref
  {eq:genGrushin-Matrix-simplified}. Further, let $\Omega\subseteq \R^d$
  be an open subset such that $\Omega\cap \{x_1=0\}\neq \emptyset$. If
  $\frac{2}{p^{\mbox{}_{\prime}}}\le \alpha<2$, then for
  every $f\in L^1_{loc}(\Omega)$ with
  $\abs{\nabla_{\!\bm{A}}f}_{\bm{A}}\in L^p(\Omega)$, one has that
  \begin{displaymath}
    \nabla_{\!\bm{A}}(f\mathds{1}_{\Omega^{\pm}})=(\nabla_{\!\bm{A}}f)\mathds{1}_{\Omega^{\pm}}
  \end{displaymath}
  with
  \begin{displaymath}
    \abs{\nabla_{\!\bm{A}}(f\mathds{1}_{\Omega^{\pm}})}_{\bm{A}} \in L^p(\Omega^{\pm}).
  \end{displaymath} 
  In particular, the functional $\E$ given by \eqref{eq:Grushin-p-energy-simply}
  satisfies
\begin{equation}
  \label{eq:2}
  \E(f)= \E(f\mathds{1}_{\Omega^{+}}) +\E(f\mathds{1}_{\Omega^{-}})
\end{equation}
 In particular, one has that
\begin{displaymath}
   \langle \E'(f),g\rangle= \langle \E'(f\mathds{1}_{\Omega^{+}}),g\rangle+\langle \E'(f\mathds{1}_{\Omega^{-}}),g\rangle
\end{displaymath}
for every $f$, $g\in L^1_{loc}(\Omega)$ with $\abs{\nabla_{\!\bm{A}}f}_{\bm{A}}$, $\abs{\nabla_{\!\bm{A}}g}_{\bm{A}}\in L^p(\Omega)$.
\end{lem}

\begin{rem}
  \label{rem:treshold-2-comment}
  Let $(\R^d,\bm{A})$ be the Riemannian structure given by \eqref
  {eq:genGrushin-Matrix-simplified}. Then, it is worth noting that in
  Lemma~\ref{lem:5}, the factor $2$ in the lower bound $2/p'$ on $\alpha$
  results from the \emph{square root} $\bm{B}(x)$ of $\bm{A}(x)$ and not
  since the function $\E$ given by \eqref{eq:Grushin-p-energy-simply} is
  defined by with respect to the $\Euc$-norm; that is,
  \begin{displaymath}
    \E(f)=
  \tfrac{1}{p}\int_{\Omega}\labs{(|x_1|^{\frac{\alpha}{2}}
    \tfrac{\partial f}{\partial x_1}, \abs{x_1}^{\frac{\beta_{1}}{2}}
    \tfrac{\partial f}{\partial y_1},\dots,\abs{x_1}^{\frac{\beta_{m}}{2}}
    \tfrac{\partial f}{\partial y_m})}^p_{\Euc}\,
  \dx .
  \end{displaymath}
\end{rem}

\begin{proof}[Proof of Lemma~\ref{lem:5}]
  Note, it is sufficient to establish \eqref{eq:2} for
  $f \in L^{\infty}(\Omega)$ with compact $\supp(f)\subseteq \Omega$ and
  $\abs{\nabla_{\!\bm{A}}f}_{\bm{A}}\in L^p(\Omega)$. To see this, let
  for every $n>0$, $T_{n} : \R\to \R$ be the truncator given by
\begin{displaymath}
  T_{n}(s)=
  \begin{cases}
  s & \text{if $\abs{s}\le n$,}\\
  n\frac{s}{\abs{s}} &\text{otherwise,}
  \end{cases}
\end{displaymath}
for every $s\in \R$. Then, $T_n(f) \in L^{\infty}(\Omega)$, and since
$T_{n}$ is Lipschitz continuous with constant 1, $T_n(0)=0$,
$\abs{T_n(f)(x)}\le \max\{\abs{f(x)},n\}$, and since
\begin{displaymath}
  \nabla_{\!\bm{A}}T_{n}(f)=\bm{A}\nabla f\,\mathds{1}_{\{\abs{f}\le n\}}
\end{displaymath}
for every $f \in L^{\infty}(\Omega)$ with
$\abs{\nabla_{\!\bm{A}}f}_{\bm{A}}\in L^p(\Omega)$, it follows from
Lebesgue's dominated convergence theorem that $T_nf\to f$ in
$L^1_{loc}(\Omega)$, and 
\begin{displaymath}
  \nabla_{\!\bm{A}}T_{n}(f)\to \nabla_{\!\bm{A}}f\qquad\text{in $L^p(\Omega)$.}
\end{displaymath}
Next, for every $k\in \N$, let
$\Omega_k:=\{x\in
\Omega\,\vert\,\dist(x,\partial\Omega)>\frac{1}{k}\}\cap \{|x|<k\}$.
Then, $\Omega_k\subseteq \Omega_{k+1}$,
$\bigcup_{k\in \N}\Omega_k=\Omega$. Without loss of generality, we may
assume that $\Omega_1\neq \emptyset$. Then, let
$\rho\in C^{\infty}_c(\Omega)$ satisfy $0\le \rho\le 1$, $\rho\equiv 1$
on the closure $\overline{\Omega_{1}}$,
$\supp(\rho)\subseteq \Omega_{2}$, and
$\norm{\rho}_{L^1(\Omega)}=1$. Further, for every $k\ge 1$, set
$\rho_k(x):=\rho(x/k)$ for every $x\in \Omega$. Then, one has that
$\rho_k\in C^{\infty}_c(\Omega)$, $0\le \rho_k\le 1$, $\rho_k\equiv 1$
on the closure $\overline{\Omega_{k}}$, and
$\supp(\rho_k)\subseteq \Omega_{2k}$. Hence, the test-function
$T_n(f)\rho_k\in L^{\infty}(\Omega)$ and has compact support in
$\Omega$. Moreover, for every compact subset $K$ of $\Omega$, there is a
$k_0\in \N$ such that $K\subseteq \Omega_k$ for all $k\ge 0$, and so
$T_n(f)\rho_k=T_n(f)$ for all $k\ge k_0$ and $n\ge 1$. Therefore,
$T_n(f)\rho_k\to f$ in $L^1_{loc}(\Omega)$ as $n$, $k\to
\infty$. In addition, 
\begin{displaymath}
  \nabla_{\!\bm{A}}\big[T_{n}(f)\rho_k\big]=\bm{A}\nabla
  f\,\mathds{1}_{\{\abs{f}\le n\}}\, \rho_k+\frac{1}{k}\bm{A}\nabla\rho(\cdot/k)\,T_{n}(f)
\end{displaymath}
for every $f \in L^{\infty}(\Omega)$ with
$\abs{\nabla_{\!\bm{A}}f}_{\bm{A}}\in L^p(\Omega)$. Thus,
\begin{displaymath}
  \nabla_{\!\bm{A}}\big[T_{n}(f)\rho_k\big]\to \nabla_{\!\bm{A}}f\qquad\text{in $L^p(\Omega)$}
\end{displaymath}
as $n$, $k\to \infty$.

Now, let $f \in L^{\infty}(\Omega)$ with compact 
$\supp(f)\subseteq \Omega$ and
$\abs{\nabla_{\!\bm{A}}f}_{\bm{A}}\in L^p(\Omega)$. Further, let
$(\chi_n)_{n\ge 1}$ be a sequence of truncator functions given by
  \begin{displaymath}
    \chi_{n}(x_1,y_1,\dots,y_m)=
    \begin{cases}
        0 & \quad \text{if $x_1\le \frac{1}{n}$,}\\
  \frac{\log x_1 + \log n}{\log n} &\quad  \text{if $\frac{1}{n}<x_1\le 1$,}\\
  1 & \quad \text{otherwise,}    
      \end{cases}
  \end{displaymath}
  for every $x=(x_1, y_1, \dots, y_m)\in \R^d$. Then, $\chi_{n}(x)\to
  \mathds{1}_{\Omega^{+}}(x)$ as $n\to \infty$ for every $x\in
  \Omega$, and by
  Lebesgue's dominated convergence theorem, $f\chi_{n}\to f
  \mathds{1}_{\Omega^{+}}$ in $L^{1}(\Omega)$. In addition, one has that
  \begin{align*}
    \tfrac{\partial}{\partial x_1}\left(f\chi_{n}\right)
    &= \tfrac{\partial f}{\partial x_1}\,\chi_{n}+\frac{1}{x_1\,\log n}\,f\,
    \,\mathds{1}_{\Omega\cap \{\frac{1}{n}<x_1<1\}},\\
    \tfrac{\partial}{\partial y_i}\left(f\chi_{n}\right)
    &= \tfrac{\partial f}{\partial y_i}\,\chi_{n},
  \end{align*}
 and so,
 \begin{displaymath}
    \lim_{n\to\infty}\nabla_{\!\bm{A}}\left(f\chi_{n}\right)(x)
    =(\nabla_{\!\bm{A}}f)(x)\mathds{1}_{\Omega^+}(x)
 \end{displaymath}
 for a.e. $x\in \R^d$. Next, let $n\ge 2$. Then,
 \allowdisplaybreaks
  \begin{align*}
    \E(f\chi_{n})&=
  \tfrac{1}{p}\int_{\Omega}\left[ 
        |x_1|^{\alpha}\labs{\tfrac{\partial }{\partial x_1}\left(f\chi_{n}\right)}^2
        +\sum_{j=1}^{m}\abs{x_1}^{\beta_{j}}\labs{\tfrac{\partial}{\partial
                   y_j}\left(f\chi_{n}\right)}^2\right]^{\frac{p}{2}}\!\dx\\ 
&=
  \tfrac{1}{p}\int_{\Omega}\left[ 
        |x_1|^{\alpha}\labs{\tfrac{\partial f}{\partial x_1}\,\chi_{n}+\frac{1}{x_1\,\log n}\,f\,
    \,\mathds{1}_{\Omega\cap \{\frac{1}{n}<x_1<1\}}}^2\right.\\
&\mbox{}\hspace{6cm}\left. 
+\sum_{j=1}^{m}\abs{x_1}^{\beta_{j}}\labs{\tfrac{\partial f}{\partial
  y_i}\,\chi_{n}}^2\right]^{\frac{p}{2}}\!\dx\\
&\le
  \tfrac{1}{p}\int_{\Omega}\left[ 
        |x_1|^{\alpha}2\labs{\tfrac{\partial f}{\partial x_1}}^2+2\,\labs{\frac{|x_1|^{\frac{\alpha}{2}}}{x_1\,\log n}\,f\,
    \,\mathds{1}_{\Omega\cap \{\frac{1}{n}<x_1<1\}}}^2 \right.\\
&\mbox{}\hspace{6.5cm}\left.+\sum_{j=1}^{m}\abs{x_1}^{\beta_{j}}\labs{\tfrac{\partial f}{\partial
  y_i}}^2\right]^{\frac{p}{2}}\!\dx\\
&\le
  \tfrac{C_p}{p}\int_{\Omega} \left[
        |x_1|^{\alpha\frac{p}{2}}2^{\frac{p}{2}}\labs{\tfrac{\partial
  f}{\partial x_1}}^p+2^{\frac{p}{2}}
  \,\labs{\frac{|x_1|^{\frac{\alpha}{2}}}{x_1\,\log n}\,f\,
    \,\mathds{1}_{\Omega\cap \{\frac{1}{n}<x_1<1\}}}^p \right.\\
&\mbox{}\hspace{6.5cm}\left.
+\sum_{j=1}^{m}\abs{x_1}^{\beta_{j}\frac{p}{2}}\labs{\tfrac{\partial f}{\partial
  y_i}}^p\right]\,\dx\\
&=
  \tfrac{C_p2^{\frac{p}{2}}}{p}\left[\int_{\Omega} 
        |x_1|^{\alpha\frac{p}{2}}\labs{\tfrac{\partial
  f}{\partial x_1}}^p\dx + \int_{\Omega\cap \{\frac{1}{n}<x_1<1\}} 
  \frac{|x_1|^{\frac{\alpha}{2}p}}{|x_1|^p\,(\log n)^p}\,|f|^p\,
     \dx\right]\\
&\mbox{}\hspace{6.1cm}
+\tfrac{C_p}{p}\int_{\Omega}\sum_{j=1}^{m}\abs{x_1}^{\beta_{j}\frac{p}{2}}\labs{\tfrac{\partial f}{\partial
  y_i}}^p\,\dx\\
&=
  \tfrac{C_p2^{\frac{p}{2}}}{p}\left[\int_{\Omega} 
        |x_1|^{\alpha\frac{p}{2}}\labs{\tfrac{\partial
  f}{\partial x_1}}^p\dx + \int_{\frac{1}{n}}^{1}\int_{\Omega\cap \{x_1=r\}} 
  \frac{|r|^{\frac{\alpha}{2}p}}{r^p\,(\log n)^p}\,|f|^p\,
     \dy\,\td r\right]\\
&\mbox{}\hspace{6.1cm}
+\tfrac{C_p}{p}\int_{\Omega}\sum_{j=1}^{m}\abs{x_1}^{\beta_{j}\frac{p}{2}}\labs{\tfrac{\partial f}{\partial
  y_i}}^p\,\dx\\
&\le 
  \tfrac{C_p2^{\frac{p}{2}}}{p}\left[\int_{\Omega} 
        |x_1|^{\alpha\frac{p}{2}}\labs{\tfrac{\partial
  f}{\partial x_1}}^p\dx + \frac{|\supp(f)|\,\norm{f}_{\infty}}{(\log n)^p}\int_{\frac{1}{n}}^{1} 
  |r|^{(\frac{\alpha}{2}-1)p}\,\td r\right]\\
&\mbox{}\hspace{6.1cm}
+\tfrac{C_p}{p}\int_{\Omega}\sum_{j=1}^{m}\abs{x_1}^{\beta_{j}\frac{p}{2}}\labs{\tfrac{\partial f}{\partial
  y_i}}^p\,\dx 
  \end{align*}
  Note, for $p>1$ and $p-p\frac{\alpha}{2}-1\le 0$, or, equivalently,
  $\alpha\ge \frac{2}{p^{\mbox{}_{\prime}}}$, one has that
  \begin{equation}
  \label{eq:223}
    \lim_{n\to \infty}\frac{1}{(\log n)^p}\int_{\frac{1}{n}}^{1}
    r^{p(\frac{\alpha}{2}-1)}\,\td r=0.
  \end{equation}
  If $p=1$, then~\eqref{eq:223} holds provided $\alpha>0$. Therefore, we
  have shown that
  \begin{displaymath}
    \E(f\chi_{n})\to \E(f\mathds{1}_{\Omega^+})\qquad\text{as $n\to\infty$.}
  \end{displaymath}
  Similarly, let
$(\hat{\chi}_n)_{n\ge 1}$ be a sequence of truncator functions given by
  \begin{displaymath}
    \hat{\chi}_{n}(x_1,y_1,\dots,y_m)=
    \begin{cases}
       1 & \quad \text{if $x_1\le -1$,}\\
     \frac{\log (-x_1) + \log n}{\log n} &\quad  \text{if $-1<x_1\le -\frac{1}{n}$,}\\
        0 & \quad \text{if $x_1\ge -\frac{1}{n}$,} 
      \end{cases}
  \end{displaymath}
  for every $x=(x_1, y_1, \dots, y_m)\in \R^d$. Then, a similar
  computation shows that 
  \begin{displaymath}
    \E(f\hat{\chi}_{n})\to \E(f\mathds{1}_{\Omega^-})\qquad\text{as $n\to\infty$}
  \end{displaymath}
  provided $p>1$ and $\alpha>\frac{2}{p^{\mbox{}_{\prime}}}$.  Now,  since
  \begin{displaymath}
    \E(f\hat{\chi}_{n}+f
    \chi_{n})=\E(f\hat{\chi}_{n})+\E(f \chi_{n}),
  \end{displaymath}
   sending $n\to \infty$ yields that \eqref{eq:2}. Next, let $f$, $g \in L^{\infty}(\Omega)$ with compact 
$\supp(f)\subseteq \Omega$, $\supp(g)\subseteq \Omega^+$, and
, $\abs{\nabla_{\!\bm{A}}g}_{\bm{A}}\in L^p(\Omega)$. Then,
   \begin{align*}
       \frac{\E(f+tg)-\E(f)}{t}&=\frac{\E((f+tg)\mathds{1}_{\Omega^+})-\E(f\mathds{1}_{\Omega^+})}{t}\\
       &\hspace{2cm}+\frac{\E((f+tg)\mathds{1}_{\Omega^-})-\E(f\mathds{1}_{\Omega^-})}{t}\\
       &=\frac{\E((f+tg)\mathds{1}_{\Omega^+})-\E(f\mathds{1}_{\Omega^+})}{t}
   \end{align*}
 for every $t>0$, from where we can conclude that
 \begin{displaymath}
     \langle \E'(f),g\mathds{1}_{\Omega^{+}}\rangle= \langle \E'(f\mathds{1}_{\Omega^{+}}),g\mathds{1}_{\Omega^{+}}\rangle
 \end{displaymath}
 Similarly, one obtains that
 \begin{displaymath}
     \langle \E'(f),g\mathds{1}_{\Omega^{-}}\rangle= \langle \E'(f\mathds{1}_{\Omega^{-}}),g\mathds{1}_{\Omega^{-}}\rangle
 \end{displaymath}
 for $g$ with $\supp(g)\subseteq \Omega^-$. Now, for general
 $g \in L^{\infty}(\Omega)$ with compact $\supp(g)\subseteq \Omega$ and
 $\abs{\nabla_{\!\bm{A}}g}_{\bm{A}}\in L^p(\Omega)$, one has that
\begin{align*}
    \langle \E'(f),g\rangle&=\langle
                             \E'(f),g\mathds{1}_{\Omega^{+}}\rangle
                             +\langle \E'(f),g\mathds{1}_{\Omega^{-}}\rangle\\
    &=\langle
      \E'(f\mathds{1}_{\Omega^{+}}),g\mathds{1}_{\Omega^{+}}\rangle
      +\langle \E'(f\mathds{1}_{\Omega^{-}}),g\mathds{1}_{\Omega^{-}}\rangle\\
    &=\langle \E'(f\mathds{1}_{\Omega^{+}}),g\rangle
      +\langle \E'(f\mathds{1}_{\Omega^{-}}),g\rangle,
\end{align*}
concluding the proof of this lemma.
\end{proof}

 With the above given lemma, we can now outline the proof of the main
 result of this section, namely, Theorem~\ref{thm:5}.\medskip

\begin{proof}[Proof of Theorem~\ref{thm:5}.]
  For every constant $c\in \R\setminus\{0\}$, $f=c\,\mathds{1}_{\Omega}$
  is a continuous weak solution of the Gru\v{s}in-type $p$-Laplace
  equation~\eqref{eq:13} on $\Omega$. But due to Lemma~\ref{lem:5},
  $f^+:=c\,\mathds{1}_{\Omega^+}$ and $f^-:=c\,\mathds{1}_{\Omega^-}$
  are also weak solution of the Gru\v{s}in-type $p$-Laplace
  equation~\eqref{eq:13} on $\Omega$, which are having a simple jump
  discontinuity along $y_j$ crossing $\{x_1=0\}$ for every
  $j=1, \dots, m$.
\end{proof}

\subsection{The parabolic  equation}
\label{sec:separation-parabolic}

This section is dedicated to illustrate the separation phenomenon on
parabolic problems driven by $p$-Laplace operator. For this it is worth
noting the followin comment. In the following the
homogeneous Dirichlet boundary condition on the (outer) boundary $\partial\Omega$
of a given open subset $\Omega$ of $\R^d$ could be replaced with any
other kind of boundary conditions.\medskip

Let $1<p<\infty$, and $\E^{\Omega} : L^2(\Omega)\to [0,\infty]$
be the functional realizing the negative \emph{Dirichlet $p$-Laplace
  operator} $-\Delta^{\!\bm{A},D}_{p}$ on $L^2(\Omega)$ for an open
subset $\Omega$ of $\R^d$ (cf. Example~\ref{ex:2-gen-grushin}). We
denote by $\{e^{t\Delta^{\!\bm{A},D}_{p}}\}_{t\ge 0}$ be the semigroup
generated by $\Delta^{\!\bm{A},D}_{p}$ on $L^2(\Omega)$.

Since the effective
domain $D(\Delta^{\!\bm{A},D}_{p})$ is a subset of
$W^{1,(q,p)}_{\!  \bm{A},0}(\Omega)$, functions
$f\in D(\Delta^{\!\bm{A},D}_{p})$ vanish at the boundary
$\partial\Omega$ in a \emph{weak sense}. However, even if
$\Omega\cap\{x_1=0\}\neq\emptyset$, then function
$f\in D(\Delta^{\!\bm{A},D}_{p})$ can have the property that
$\supp(f)\cap\{x_1=0\}\neq\emptyset$.\medskip

For the rest of this final section, suppose that $\Omega$ is an
open subset of $\R^d$ such that $\Omega\cap\{x_1=0\}\neq\emptyset$.
Further, set $\dot{\Omega}:=\Omega^+\dot{\cup} \Omega^-$. Then,
equivalently, one has that
$\dot{\Omega}=\Omega\setminus \{x_1=0\}$. Thus, the functional
$\E^{\dot{\Omega}} : L^2(\Omega)\to [0,\infty]$ given by
\begin{displaymath}
    \E^{\dot{\Omega}}(f):=
  \begin{cases}
  \tfrac{1}{p}\displaystyle
  \int_{\Omega}\abs{\nabla_{\!\! \bm{A}}f}_{\bm{A}}^{p}\,\dx   
 &\text{if $f\in W^{1,(2,p)}_{\!\bm{A},0}(\dot{\Omega})$,}\\
 +\infty &\text{otherwise,}
 \end{cases}
\end{displaymath}
is well-defined for every $f\in L^{2}(\Omega)$. It also follows from
Proposition~\ref{prop:2} that $\E^{\dot{\Omega}}$ is convex, proper, and
lower semi-continuous functional on $L^{2}(\Omega)$ realizing the
negative Dirichlet $p$-Laplace operator $-\dot{\Delta}^{\!\bm{A},D}_{p}$
on $L^2(\Omega)$. Let $\{e^{t\dot{\Delta}^{\!\bm{A},D}_{p}}\}_{t\ge 0}$ denote the semigroup
generated by $\dot{\Delta}^{\!\bm{A},D}_{p}$ on $L^2(\Omega)$.

Here, it is worth mentioning that in contrary to the
operator $\Delta^{\!\bm{A},D}_{p}$, functions
$f\in D(\dot{\Delta}^{\!\bm{A},D}_{p})$ vanish at the boundary
$\partial\Omega^{+}\cup \partial\Omega^-$ in a \emph{weak sense}.\medskip

 Next, we aim to show
that for the two functionals, one has that

\begin{lem}
  \label{lem:equal-energy-identity}
  Let $1<p<\infty$, $p^{\mbox{}_{\prime}}=p/(p-1)$, and for
  $0\le \alpha<2$, $\beta_1$, \dots, $\beta_m \ge 0$, let
  $(\R^d,\bm{A})$ be the Riemannian structure given by \eqref
  {eq:genGrushin-Matrix-simplified}. Further, let $\Omega\subseteq \R^d$
  be an open subset such that $\Omega\cap \{x_1=0\}\neq \emptyset$. If
  $\frac{2}{p^{\mbox{}_{\prime}}}\le \alpha<2$, then one has
  that the two spaces $W^{1,(q,p)}_{\!  \bm{A},0}(\Omega)$ and
  $W^{1,(q,p)}_{\!  \bm{A},0}(\dot{\Omega})$ coincide and
  \begin{equation}
    \label{eq:equality-of-EandEdotOmega}
    \E^{\dot{\Omega}}=\E^{\Omega}.
  \end{equation}
\end{lem}

\begin{proof}
  Let $f\in C^{\infty}_c(\Omega)$. Then the proof of Lemma~\ref{lem:5}
  shows that
  $f\mathds{1}_{\Omega^{\pm}}\in W^{1,(2,p)}_{\bm{A},0}(\Omega^{\pm})$,
  and since $\dot{\Omega}:=\Omega^+\dot{\cup} \Omega^-$, it follows that
  $f\in W^{1,(2,p)}_{\!\bm{A},0}(\dot{\Omega})$. Moreover, by
  Lemma~\ref{lem:5}, one has that \eqref{eq:2} for every
  $f\in C^{\infty}_c(\Omega)$. Since
  $W^{1,(2,p)}_{\!\bm{A},0}(\dot{\Omega})$ is closed, we can conclude
  that
  $W^{1,(2,p)}_{\!\bm{A},0}(\Omega)\subseteq
  W^{1,(2,p)}_{\!\bm{A},0}(\dot{\Omega})$
  and, in particular, that \eqref{eq:2} holds for every
  $f\in W^{1,(2,p)}_{\!\bm{A},0}(\Omega)$. On the other hand, one 
  immediately sees that the reverse inclusion of
  $W^{1,(2,p)}_{\!\bm{A},0}(\dot{\Omega})$ into $
  W^{1,(2,p)}_{\!\bm{A},0}(\Omega)$ always 
  holds. Thus, we have thereby shown that
\begin{displaymath}
    W^{1,(2,p)}_{\!\bm{A},0}(\Omega)=W^{1,(2,p)}_{\!\bm{A},0}(\dot{\Omega})
\end{displaymath}
and by construction of $\E^{\Omega}$ and $\E^{\dot{\Omega}}$ that
\eqref{eq:equality-of-EandEdotOmega} needs to hold.
\end{proof}

Now, the following result, is our main result of this final section.

\begin{thm}
  \label{thm:semigroup-identity}
  Let $1<p<\infty$, $p^{\mbox{}_{\prime}}=p/(p-1)$, and for
  $0\le \alpha<2$, $\beta_1$, \dots, $\beta_m \ge 0$, let
  $(\R^d,\bm{A})$ be the Riemannian structure given by \eqref
  {eq:genGrushin-Matrix-simplified}. Further, let $\Omega\subseteq \R^d$
  be an open subset such that $\Omega\cap \{x_1=0\}\neq \emptyset$. If
  $\frac{2}{p^{\mbox{}_{\prime}}}\le \alpha<2$, then for every $f\in
  L^2(\Omega)$, one has that 
  \begin{equation}
  \label{eq:2-semigroups-subsets}
  e^{t\Delta_p^{\!\bm{A},D}}[f\mathds{1}_{\Omega^{\pm}}]=
  [e^{t\Delta_p^{\!\bm{A},D}}f]\mathds{1}_{\Omega^{\pm}}
\end{equation}
and
\begin{equation}
  \label{eq:2-semigroups-eq}
  e^{t\Delta_p^{\!\bm{A},D}}f=
  [e^{t\Delta_p^{\!\bm{A},D}}f]\mathds{1}_{\Omega^{+}} 
  + [e^{t\Delta_p^{\!\bm{A},D}}f]\mathds{1}_{\Omega^{-}}
\end{equation}
for every $t>0$.
\end{thm}

\begin{proof}
 Let $f\in L^2(\Omega)$. Due to the immediate smoothing
effect~\eqref{eq:domain-smoothing-effect-semigroup}, one has that 
\begin{displaymath}
 e^{t\dot{\Delta}_p^{\!\bm{A},D}}f\text{ and }
 e^{t\dot{\Delta}_p^{\!\bm{A},D}}[f\mathds{1}_{\Omega^{\pm}}]\in W^{1,(2,p)}_{\!\bm{A},0}(\dot{\Omega})
\end{displaymath}
for every $t>0$. Hence and by the proof of Lemma~\ref{lem:5}, one has that
  \begin{displaymath}
    \left[e^{t\dot{\Delta}_p^{\!\bm{A},D}}f\right]\mathds{1}_{\Omega^{\pm}}\in W^{1,(2,p)}_{\bm{A},0}(\Omega^{\pm}).
\end{displaymath}
For keeping the next computation clear, we set
\begin{displaymath}
  u(t)=e^{t\dot{\Delta}_p^{\!\bm{A},D}}[f\mathds{1}_{\Omega^{\pm}}]\quad\text{and}\quad
  v(t)=e^{t\dot{\Delta}_p^{\!\bm{A},D}}f
\end{displaymath}
for every $t\ge 0$. Then, one sees that
\allowdisplaybreaks
\begin{align*}
  &\frac{\td}{\dt}\tfrac{1}{2}\norm{u(t)-v(t)\mathds{1}_{\Omega^{\pm}}}_{L^2(\Omega)}^2\\
   &\quad 
    = \left\langle
     \dot{\Delta}_p^{\!\bm{A},D}u(t)
     -\left[\dot{\Delta}_p^{\!\bm{A},D}v(t)\right]\mathds{1}_{\Omega^{\pm}},
     u(t)-v(t)\mathds{1}_{\Omega^{\pm}}
     \right\rangle_{L^2(\Omega)} \\
  &\quad 
    = \left\langle
     \dot{\Delta}_p^{\!\bm{A},D}u(t),
     u(t)-v(t)\mathds{1}_{\Omega^{\pm}}
     \right\rangle_{L^2(\Omega)} \\
  &\mbox{}\hspace{2.5cm} 
    - \left\langle \dot{\Delta}_p^{\!\bm{A},D}v(t),
     u(t)\mathds{1}_{\Omega^{\pm}}-v(t)\mathds{1}_{\Omega^{\pm}}
     \right\rangle_{L^2(\Omega)} \\
   &\quad
    = -\int_{\Omega}\labs{\nabla_{\!\bm{B}}u(t)}_{\Euc}^{p-2}
    \langle
     \nabla_{\!\bm{B}}u(t),\nabla_{\!\bm{B}}(u(t)-v(t)\mathds{1}_{\Omega^{\pm}})\rangle_{\Euc}\,\dx\\
  &\mbox{}\hspace{2cm}
    +\int_{\Omega}\labs{\nabla_{\!\bm{B}}v(t)}_{\Euc}^{p-2}
    \langle
     \nabla_{\!\bm{B}}v(t),\nabla_{\!\bm{B}}(u(t)\mathds{1}_{\Omega^{\pm}}-v(t)\mathds{1}_{\Omega^{\pm}})\rangle_{\Euc}\,\dx\\
   &\quad
    = -\int_{\Omega}\labs{\nabla_{\!\bm{B}}u(t)}_{\Euc}^{p-2}
    \langle
     \nabla_{\!\bm{B}}u(t),\nabla_{\!\bm{B}}(u(t)-v(t)\mathds{1}_{\Omega^{\pm}})\rangle_{\Euc}\,\dx\\
  &\mbox{}\hspace{1.5cm}
    +\int_{\Omega}\labs{\nabla_{\!\bm{B}}v(t)\mathds{1}_{\Omega^{\pm}}}_{\Euc}^{p-2}
    \langle
     \nabla_{\!\bm{B}}v(t)\mathds{1}_{\Omega^{\pm}},\nabla_{\!\bm{B}}(u(t)-v(t)\mathds{1}_{\Omega^{\pm}})\rangle_{\Euc}\,\dx\\
  &\quad \le 0
\end{align*}
for a.e. $t>0$, where used the monotonicity of the power function
$x\mapsto \abs{x}^{p-2}_{\Euc}x$. From this, we can conclude that 
\begin{displaymath}
  e^{t\dot{\Delta}_p^{\!\bm{A},D}}[f\mathds{1}_{\Omega^{\pm}}]=\left[e^{t\dot{\Delta}_p^{\!\bm{A},D}}f\right]
  \mathds{1}_{\Omega^{\pm}}
\end{displaymath}
for every $t\ge 0$, and as by Lemma~\ref{lem:equal-energy-identity}, the
functionals $\E^{\Omega}$ and $\E^{\dot{\Omega}}$ satisfy
\eqref{eq:equality-of-EandEdotOmega}, the two semigroups
$\{e^{t\Delta^{\!\bm{A},D}_{p}}\}_{t\ge 0}$ and
$\{e^{t\dot{\Delta}^{\!\bm{A},D}_{p}}\}_{t\ge 0}$ coincide. From this,
we can conclude the statements of this theorem.
\end{proof}
%
%
%
%


\end{document}